\documentclass[11pt,a4paper]{article}

\usepackage{graphicx,latexsym,euscript,makeidx,color,bm}
\usepackage{amsmath,amsfonts,amssymb,amsthm,thmtools,mathtools,mathrsfs,enumerate}
\usepackage[colorlinks,linkcolor=blue,citecolor=red]{hyperref}

\usepackage{geometry}
\allowdisplaybreaks[4]
  \def\cA{{\cal A}}  
  \def\cB{{\cal B}}  
 \def\sC{\mathscr{C}} \def\cC{{\cal C}}  
  \def\cD{{\cal D}}  
\def\dbE{\mathbb{E}}    
\def\dbF{\mathbb{F}}  \def\cF{{\cal F}}  
  \def\cG{{\cal G}}  
\def\dbH{\mathbb{H}}    
  \def\cI{{\cal I}}  
 \def\sJ{\mathscr{J}} \def\cJ{{\cal J}}  
  \def\cK{{\cal K}}  
  \def\cL{{\cal L}}

\def\dbP{\mathbb{P}}  \def\cP{{\cal P}}  
  \def\cQ{{\cal Q}}  
\def\dbR{\mathbb{R}} \def\sR{\mathscr{R}} \def\cR{{\cal R}}  
\def\dbS{\mathbb{S}} \def\sS{\mathscr{S}} \def\cS{{\cal S}}  
    
 \def\sU{\mathscr{U}} \def\cU{{\cal U}}  
 \def\sV{\mathscr{V}} \def\cV{{\cal V}}  
    
  \def\cX{{\cal X}}

\def\ss{\smallskip}   \def\lt{\left}        \def\hb{\hbox}
\def\ms{\medskip}     \def\rt{\right}       \def\ae{\text{a.e.}}
      \def\lan{\langle}     
\def\q{\quad}         \def\ran{\rangle}     \def\tr{\hb{tr$\,$}}
\def\qq{\qquad}           
\def\no{\noindent}        
\def\hp{\hphantom}    \def\blan{\big\lan}   
\def\nn{\nonumber}    \def\bran{\big\ran}   
\def\rf{\eqref}       \def\Blan{\Big\lan\!} 
\def\cd{\cdot}        \def\Bran{\!\Big\ran} 
\def\deq{\triangleq}  \def\({\Big(}         
\def\les{\leqslant}   \def\){\Big)}         
\def\ges{\geqslant}   \def\[{\Big[}         \def\ts{\textstyle}
\def\ti{\tilde}       \def\]{\Big]}         \def\im{\hb{im\,}}
   \def\Lan{[\![}        
\def\Ra{\Rightarrow}  \def\Ran{]\!]}        

\def\a{\alpha}        \def\Om{\Omega}  
\def\b{\beta}   \def\D{\Delta}   \def\d{\delta}   \def\F{\Phi}     
\def\z{\zeta}       \def\Si{\Sigma}  
\def\f{\varphi}   \def\l{\lambda}  \def\m{\mu}      \def\e{\varepsilon}
    \def\i{\infty}      
\newtheoremstyle{thry}
{}      
{}      
{\sl}   
{}      
{\bf}   
{.}     
{.5em}  
{}      
\theoremstyle{thry}

\newtheorem{theorem}{Theorem}[section]
\newtheorem{proposition}[theorem]{Proposition}
\newtheorem{corollary}[theorem]{Corollary}
\newtheorem{lemma}[theorem]{Lemma}

\theoremstyle{definition}

\theoremstyle{remark}
\newtheorem{remark}[theorem]{Remark}

\def\punct{}
\newtheoremstyle{dotless}{}{}{\rm}{}{\bf}{\punct}{.5em}{}
\theoremstyle{dotless}
\newtheorem{taggedthmx}{}
\newenvironment{taggedthm}[1]
 {\renewcommand\thetaggedthmx{#1}\taggedthmx}
 {\endtaggedthmx}
\newtheorem{taggedassumptionx}{}
\newenvironment{taggedassumption}[1]
 {\renewcommand\thetaggedassumptionx{#1}\taggedassumptionx}
 {\endtaggedassumptionx}

\makeatletter
   
   \@addtoreset{equation}{section}
   \newcommand{\setword}[2]{%
   \phantomsection
   #1\def\@currentlabel{\unexpanded{#1}}\label{#2}%
   }
\makeatother

\begin{document}

\title{\bf Indefinite Backward Stochastic Linear-Quadratic Optimal Control Problems}
\author{Jingrui Sun\thanks{Department of Mathematics, Southern University of Science and Technology, Shenzhen,
                           518055, China (Email: {\tt sunjr@sustech.edu.cn}).
                           This author is supported by NSFC Grant 11901280, Guangdong Basic and Applied Basic Research Foundation 2021A1515010031,
                           and SUSTech start-up funds Y01286128 and Y01286228. }
\and
Zhen Wu\thanks{Corresponding Author. School of Mathematics, Shandong University, Jinan 250100, China (Email: {\tt wuzhen@sdu.edu.cn}).
               This author is supported by NSFC Grants 11831010, 61961160732 and Shandong Provincial Natural Science
               Foundation ZR2019ZD42.}
\and
Jie Xiong\thanks{Department of Mathematics and SUSTech International center for Mathematics, Southern University
                 of Science and Technology, Shenzhen, 518055, China (Email: {\tt xiongj@sustech.edu.cn}).
                 This author is supported by NSFC Grants 61873325 and 11831010, and SUSTech start-up funds Y01286120 and Y01286220.}}

\maketitle

\no{\bf Abstract.}
This paper is concerned with a backward stochastic linear-quadratic (LQ, for short) optimal control problem
with deterministic coefficients.
The weighting matrices are allowed to be indefinite, and cross-product terms in the control and state processes
are present in the cost functional.
Based on a Hilbert space method, necessary and sufficient conditions are derived for the solvability of the problem,
and a general approach for constructing optimal controls is developed.
The crucial step in this construction is to establish the solvability of a Riccati-type equation,
which is accomplished under a fairly weak condition by investigating the connection with forward stochastic LQ
optimal control problems.

\ms
\no{\bf Key words.}
indefinite, backward stochastic differential equation, linear-quadratic, optimal control, Riccati equation.

\ms
\no{\bf AMS 2020 Mathematics Subject Classification.} 93E20, 49N10, 49N35, 49K27.

\section{Introduction}\label{Sec:Intro}

Let $(\Om,\cF,\dbP)$ be a complete probability space on which a standard one-dimensional
Brownian motion $W=\{W(t);t\ges0\}$ is defined, and let $\dbF=\{\cF_t\}_{t\ges0}$ be the
usual augmentation of the natural filtration generated by $W$.
For a random variable $\xi$, we write $\xi\in\cF_t$ if $\xi$ is $\cF_t$-measurable;
and for a stochastic process $\f$, we write $\f\in\dbF$ if it is progressively measurable with
respect to the filtration $\dbF$.

\ms

Consider the following controlled linear backward stochastic differential equation
(BSDE, for short) over a finite horizon $[0,T]$:
\begin{equation}\label{state}\left\{\begin{aligned}
dY(t) &= [A(t)Y(t)+B(t)u(t)+C(t)Z(t)]dt + Z(t)dW(t), \\
 Y(T) &= \xi,
\end{aligned}\right.\end{equation}
where the {\it coefficients} $A,C:[0,T]\to\dbR^{n\times n}$ and $B:[0,T]\to\dbR^{n\times m}$ of the
{\it state equation} \rf{state} are given bounded deterministic functions;
the {\it terminal value} $\xi$ is in $L^2_{\cF_T}(\Om;\dbR^n)$, the space of $\dbR^n$-valued, $\cF_T$-measurable,
square-integrable random variables;
and $u$, valued in $\dbR^m$, is the {\it control process}. The class of {\it admissible controls} for \rf{state} is
$$\ts \sU=\Big\{ u:[0,T]\times\Om\to\dbR^m \bigm| u\in\dbF~\hb{and}~\dbE\int_0^T|u(t)|^2dt<\i \Big\}, $$
and the associated {\it cost} is given by the following quadratic functional:
\begin{align}\label{cost}
J(\xi;u) &= \dbE\bigg[\lan GY(0),Y(0)\ran + \!\int_0^T\!\Blan
          \begin{pmatrix}  Q(t) & \!\!S_1^\top(t) & \!\!S_2^\top(t)\\
                         S_1(t) & \!\!R_{11}(t)   & \!\!R_{12}(t)  \\
                         S_2(t) & \!\!R_{21}(t)   & \!\!R_{22}(t)  \end{pmatrix}\!\!
          \begin{pmatrix}Y(t) \\ Z(t) \\ u(t)\end{pmatrix}\!,\!
          \begin{pmatrix}Y(t) \\ Z(t) \\ u(t)\end{pmatrix}\Bran dt \bigg],
\end{align}
where the superscript $\top$ denotes the transpose of a matrix, $G$ is a symmetric $n\times n$ constant matrix, and
$$ Q, \q S = \begin{pmatrix}S_1 \\ S_2 \end{pmatrix},
\q R = \begin{pmatrix}R_{11} & R_{12} \\
                      R_{21} & R_{22} \end{pmatrix} $$
are bounded deterministic matrix-valued functions of proper dimensions over $[0,T]$ such that
the blocked matrix in the cost functional is symmetric.
The optimal control problem of interest in the paper can be stated as follows.

\begin{taggedthm}{Problem (BSLQ).}
For a given terminal state $\xi\in L^2_{\cF_T}(\Om;\dbR^n)$, find a control $u^*\in\sU$ such that
\begin{align}\label{def-u*}
J(\xi;u^*)=\inf_{u\in\sU}J(\xi;u) \equiv V(\xi).
\end{align}
\end{taggedthm}

The above problem is usually referred to as a {\it backward stochastic linear-quadratic (LQ, for short) optimal
control problem} (BSLQ problem, for short), due to the linearity of the backward state equation \rf{state} and
the quadratic form of the cost \rf{cost}.
A process $u^*\in\sU$ satisfying \rf{def-u*} is called an {\it optimal control} for the terminal state $\xi$;
the adapted solution $(Y^*,Z^*)$ of the state equation \rf{state} corresponding to $u=u^*$ is called an
{\it optimal state process};
and the function $V$ is called the {\it value function} of Problem (BSLQ).

\ms

The study of optimal control for BSDEs is important and appealing, not only at the theoretical level, but also
in financial applications.
It is well known that BSDEs play a central role in stochastic control theory and have had fruitful applications
in finance (see, e.g., \cite{Ma-Yong1999,Yong-Zhou1999,Pham2009,Peng2010,Zhang2017}).
An optimal control problem for backward differential equations arises naturally when we look at a two-person
zero-sum differential game from a leader-follower point of view (see, e.g., \cite{Yong2002} and \cite[Chapter 6]{Yong2015}).
In the case of stochastic differential games, the state equation involved in the backward optimal control
problem becomes a BSDE, and it turns out that we need to handle a backward stochastic optimal control problem.
On the other hand, in mathematical finance one frequently encounters financial investment problems with future
conditions specified, for example, the mean-variance problem for a given level of expected return.
Such kind of problems usually can be modeled as an optimal control problem with terminal constraints, and under
certain conditions, it can be converted into a backward stochastic optimal control problem;
see the recent work of Bi--Sun--Xiong \cite{Bi-Sun-Xiong2020}.

\ms

The LQ optimal control problem for BSDEs was initially investigated by Lim--Zhou \cite{Lim-Zhou2001},
where the coefficients are deterministic, {\it no cross terms} in $(Y,Z,u)$ appears in the quadratic
cost functional, and all the weighting matrices are {\it positive semidefinite}.
They obtained a complete solution for such a BSLQ problem, using a forward formulation and a limiting procedure,
together with the completion-of-squares technique.
Along this line, a couple of follow-up works appeared afterward.
For example, Wang--Wu--Xiong \cite{Wang-Wu-Xiong2012} studied a one-dimensional BSLQ problem under partial information;
Li--Sun--Xiong \cite{Li-Sun-Xiong2019} generalized the results of Lim--Zhou \cite{Lim-Zhou2001}
to the case of mean-field backward LQ optimal control problems;
Sun--Wang \cite{Sun-Wang2020} further carried out a thorough investigation on the backward stochastic LQ problem
with random coefficients;
Huang--Wang--Wu \cite{Huang-Wang-Wu2016} studied a backward mean-field linear-quadratic-Gaussian
game with full/partial information;
Wang--Xiao--Xiong \cite{Wang-Xiao-Xiong2018} analyzed a kind of LQ nonzero-sum differential game with
asymmetric information for BSDEs;
Du--Huang--Wu \cite{Du-Huang-Wu2018} considered a dynamic game of $N$ weakly-coupled
linear BSDE systems involving mean-field interactions;
and based on \cite{Lim-Zhou2001,Li-Sun-Xiong2019}, Bi--Sun--Xiong \cite{Bi-Sun-Xiong2020} developed
a theory of optimal control for controllable stochastic linear systems.
It is worthy to point out that the above mentioned works depend crucially on the positive/nonnegative
definiteness assumption imposed on the weighting matrices, and most of them do not allow cross terms
in $(Y,Z,u)$ in cost functionals.
However, one often encounters situations where the positive/nonnegative definiteness assumption is not
fulfilled. For example, let us look at the following maximin control problem (this kind of problems arises
in the study of zero-sum differential games):
$$ \max_{v\in L_\dbF^2(0,1;\dbR)}\min_{u\in L_\dbF^2(0,1;\dbR)}
   \dbE\bigg\{|X(1)|^2 + 2\xi X(1) + \int_0^1 \[|u(t)|^2-(a^2+1)|v(t)|^2\]dt\bigg\} $$
subject to
$$\left\{\begin{aligned}
dX(t) &= u(t)dt + [X(t)+v(t)]dW(t), \q t\in[0,1], \\
 X(0) &= 0,
\end{aligned}\right.$$
where $L_\dbF^2(0,1;\dbR)$ is the space of $\dbF$-progressively measurable processes $\f:[0,1]\times\Om\to\dbR$
with $\dbE\int_0^1|\f(t)|^2dt<\i$, $\xi$ is an $\cF_1$-measurable, bounded random variable, and $a>0$ is a constant.
For a given $v\in L_\dbF^2(0,1;\dbR)$, the minimization problem is a standard forward stochastic LQ optimal control problem,
to which applying the theory developed in \cite{Sun-Li-Yong2016} (see also \cite[Chapter 2]{Sun-Yong2020}),
we can easily obtain the minimum $V(\xi;v)$ (depending on $\xi$ and $v$):
$$ V(\xi;v) = \dbE\int_0^1 \[-|\eta(t)|^2+2\z(t)v(t)-a^2|v(t)|^2\]dt, $$
where $(\eta,\z)$ is the adapted solution to the BSDE
$$\left\{\begin{aligned}
d\eta(t) &= [\eta(t)-\z(t)-v(t)]dt + \z(t)dW(t), \q t\in[0,1], \\
 \eta(1) &= \xi.
\end{aligned}\right.$$
Using the transformations
$$ Y(t)=\eta(t), \q Z(t)=\z(t), \q u(t)= v(t)-{1\over a^2}\z(t), $$
we see the maximization problem is equivalent to the BSLQ problem with the state equation
$$\left\{\begin{aligned}
dY(t) &= \[Y(t)-{a^2+1\over a^2}Z(t)-u(t)\]dt + Z(t)dW(t), \q t\in[0,1], \\
 Y(1) &= \xi
\end{aligned}\right.$$
and the cost functional
$$ J(\xi;v) = \dbE\int_0^1 \[|Y(t)|^2-{1\over a^2}|Z(t)|^2 + a^2|u(t)|^2\]dt. $$
Clearly, the positive/nonnegative definiteness condition is not satisfied in this BSLQ problem since
the coefficient of $|Z(t)|^2$ is negative.

\ms

The {\it indefinite} (by which we mean the weighting matrices in the cost functional are not necessarily
positive semidefinite) backward stochastic LQ optimal control problem remains open.
It is of great practical importance and more difficult than the definite case mentioned previously.
The reason for this is that without the positive definiteness assumption one does not even know whether
the problem admits a solution.
On the other hand, even if optimal controls exist, it is by no means trivial to construct one
by using the forward formulation and limiting procedure developed in \cite{Lim-Zhou2001}, because
the forward formulation is also indefinite and it is not clear whether the solution to its associated
Riccati equation is invertible or not (and hence the limiting procedure cannot proceed).
Moreover, the presence of the cross terms in $(Y,Z,u)$ in the cost functional, especially the cross term
in $Y$ and $Z$, brings extra difficulty to the problem.
As we shall see in \autoref{sec:construction}, for the indefinite problem one cannot eliminate all cross
terms simultaneously by transformations.
We point out that taking cross terms into consideration is not just to make the framework more general,
but also to prepare for solving LQ differential games.
As mentioned earlier, a BSLQ problem arises when we look at the game in a leader-follower manner, whose
cost functional is exactly of the form introduced in this paper.

\ms

The purpose of this paper is to carry out a thorough study of the indefinite BSLQ problem.
We explore the abstract structure of Problem (BSLQ) from a Hilbert space point of view and derive
necessary and sufficient conditions for the existence of optimal controls.
The backward problem turns out to possess a similar structure as the forward stochastic LQ problem
(Problem (FSLQ), for short),
which suggests that the uniform convexity of the cost functional is the essential condition for
solving Problem (BSLQ).
We also establish a characterization of the optimal control by means of forward-backward stochastic
differential equations (FBSDEs, for short).
With this characterization the verification of the optimality of the control constructed in \autoref{sec:construction}
becomes straightforward and no longer needs the completion-of-squares technique.
The crucial step in constructing the optimal control is to establish the solvability of a Riccati-type equation.
We accomplish this by examining the connection between Problem (BSLQ) and Problem (FSLQ) and discovering some
nice properties of the solution to the Riccati equation associated with Problem (FSLQ).
It is worth pointing out that a transformation, which converts the original backward problem to an equivalent one
with $G=Q=0$, plays a central role in our analysis.
This transformation simplifies Problem (BSLQ), and more importantly, enables us to easily obtain the positivity
of the solution to the Riccati equation associated with Problem (FSLQ).
As we shall see, this positivity property is the key to employ the limiting procedure.

\ms

The remainder of this paper is structured as follows.
In \autoref{Sec:Pre} we give the preliminaries and collect some recently developed results on
forward stochastic LQ optimal control problems.
In \autoref{sec:Hilbert} we study Problem (BSLQ) from a Hilbert space point of view and derive
necessary and sufficient conditions for the existence of an optimal control.
By means of forward-backward stochastic differential equations, a characterization of the optimal
control is present in \autoref{sec:FBSDE}.
The connection between Problem (BSLQ) and Problem (FSLQ) is discussed in \autoref{sec:connection}.
In \autoref{sec:construction} we simplify Problem (BSLQ) and construct the optimal control in the case
that the cost functional is uniformly convex.
\autoref{sec:conclusion} concludes the paper.

\section{Preliminaries}\label{Sec:Pre}

We begin by introducing some notation. Let $\dbR^{n\times m}$ be the Euclidean space of $n\times m$ real matrices,
equipped with the Frobenius inner product
$$ \lan M,N\ran=\tr(M^\top N), \q M,N\in\dbR^{n\times m}, $$
where $\tr(M^\top N)$ is the trace of $M^\top N$.
The norm induced by the Frobenius inner product is denoted by $|\cd|$.
The identity matrix of size $n$ is denoted by $I_n$.
When no confusion arises, we often suppress the index $n$ and write $I$ instead of $I_n$.
Let $\dbS^n$ be the subspace of $\dbR^{n\times n}$ consisting of symmetric matrices.
For $\dbS^n$-valued functions $M$ and $N$, we write $M\ges N$ (respectively, $M>N$) if $M-N$ is positive
semidefinite (respectively, positive definite) almost everywhere (with respect to the Lebesgue measure),
and write $M\gg0$ if there exists a constant $\d>0$ such that $M\ges\d I_n$.
For a subset $\dbH$ of $\dbR^{n\times m}$, we denote by $C([0,T];\dbH)$ the space of continuous functions
from $[0,T]$ into $\dbH$, and by $L^\i(0,T;\dbH)$ the space of Lebesgue measurable, essentially bounded functions
from $[0,T]$ into $\dbH$.
Besides the space $L^2_{\cF_T}(\Om;\dbR^n)$ introduced previously, the following spaces of stochastic processes
will also be frequently used in the sequel:
\begin{align*}
L_\dbF^2(0,T;\dbH)
   &= \ts\Big\{\f:[0,T]\times\Om\to\dbH~|~\f\in\dbF~\hb{and}~\dbE\int^T_0|\f(t)|^2dt<\i \Big\}, \\
L_\dbF^2(\Om;C([0,T];\dbH))
   &= \Big\{\f:[0,T]\times\Om\to\dbH~|~\f~\hb{has continuous paths,}~\f\in\dbF,\\
   &\hp{=\Big\{\ } \ts\hb{and~}\dbE\[\sup_{0\les t\les T}|\f(t)|^2\]<\i\Big\}.
\end{align*}

\ss

We impose the following conditions on the coefficients of the state equation \rf{state}
and the weighting matrices of the cost functional \rf{cost}.

\begin{taggedassumption}{(A1)}\label{ass:A1}
The coefficients of the state equation \rf{state} satisfy
\begin{align*}
A\in L^\i(0,T;\dbR^{n\times n}), ~
B\in L^\i(0,T;\dbR^{n\times m}), ~
C\in L^\i(0,T;\dbR^{n\times n}).
\end{align*}
\end{taggedassumption}

\begin{taggedassumption}{(A2)}\label{ass:A2}
The weighting matrices in the cost functional \rf{cost} satisfy
\begin{align*}
G\in\dbS^n,~
Q\in L^\i(0,T;\dbS^{n}), ~
S\in L^\i(0,T;\dbR^{(n+m)\times n}), ~
R\in L^\i(0,T;\dbS^{n+m}).
\end{align*}
\end{taggedassumption}

For the well-posedness of the state equation \rf{state} we present the following lemma,
which is a direct consequence of the theory of linear BSDEs; see \cite[Chapter 7]{Yong-Zhou1999}.

\begin{lemma}\label{lmm:Sun2019}
Let \ref{ass:A1} hold. Then for any $(\xi,u)\in L^2_{\cF_T}(\Om;\dbR^n)\times\sU$,
the state equation \rf{state} admits a unique adapted solution
$$ (Y,Z)\in L_\dbF^2(\Om;C([0,T];\dbR^n))\times L_\dbF^2(0,T;\dbR^n). $$
Moreover, there exists a constant $K>0$, independent of $\xi$ and $u$, such that
\begin{align}\label{GUJI}
\dbE\lt[\sup_{0\les t\les T}|Y(t)|^2+\int_0^T|Z(t)|^2dt\rt]
\les K\dbE\lt[|\xi|^2+\int_0^T|u(t)|^2dt\rt].
\end{align}
\end{lemma}

We next collect some results from forward stochastic LQ optimal control theory.
These results will be needed in \autoref{sec:connection} and \autoref{sec:construction}.
Consider the forward linear stochastic differential equation
\begin{equation}\label{state:SLQ}\left\{\begin{aligned}
d\cX(t) &= [\cA(t)\cX(t)+\cB(t)v(t)]dt + [\cC(t)\cX(t)+\cD(t)v(t)]dW(t),\q t\in[0,T], \\
 \cX(0) &= x,
\end{aligned}\right.\end{equation}
and the cost functional
\begin{align}\label{cost:SLQ}
\cJ(x;v) =\dbE\Bigg[\lan\cG\cX(T),\cX(T)\ran
+\int_0^T\Blan\!\begin{pmatrix}\cQ(t) & \!\cS^\top(t) \\
                               \cS(t) & \!\cR(t)      \end{pmatrix}\!
                \begin{pmatrix}\cX(t) \\ v(t) \end{pmatrix}\!,
                \begin{pmatrix}\cX(t) \\ v(t) \end{pmatrix}\! \Bran dt \Bigg],
\end{align}
where in \rf{state:SLQ} and \rf{cost:SLQ},
\begin{eqnarray*}
& \cA,\cC \in L^\i(0,T;\dbR^{n\times n}), \q \cB,\cD \in L^\i(0,T;\dbR^{n\times m}), \\
& \cG\in\dbS^n, \q \cQ\in L^\i(0,T;\dbS^n), \q \cS\in L^\i(0,T;\dbR^{m\times n}), \q\cR\in L^\i(0,T;\dbS^m).
\end{eqnarray*}
The forward stochastic LQ optimal control problem is as follows.

\begin{taggedthm}{Problem (FSLQ)}\label{Prob:FSLQ}
For a given initial state $x\in\dbR^n$, find a control $v^*\in\sU=L_\dbF^2(0,T;\dbR^m)$ such that
\begin{equation}\label{SLQ:v*}
  \cJ(x;v^*) = \inf_{v\in\sU}\cJ(x;v) \equiv \cV(x).
\end{equation}
\end{taggedthm}

The control $v^*$ (if it exists) in \rf{SLQ:v*} is called an {\it open-loop optimal control}
for the initial state $x$, and $\cV(x)$ is called the {\it value} of Problem (FSLQ) at $x$.
Note that Problem (FSLQ) is an indefinite LQ optimal control problem, since we do not require
the weighting matrices to be positive semidefinite.
The following lemma establishes the solvability of Problem (FSLQ) under a condition that is
nearly necessary for the existence of open-loop optimal controls.
We refer the reader to Sun--Li--Yong \cite{Sun-Li-Yong2016} and the recent book \cite{Sun-Yong2020}
by Sun--Yong for proofs and further information.

\begin{lemma}\label{lmm:SLQ-main}
Suppose that there exists a constant $\a>0$ such that
\begin{equation}\label{FSLQ:cJ-uni-tu}
 \cJ(0;v)\ges \a\|v\|^2, \q\forall v\in\sU.
\end{equation}
Then the Riccati differential equation
\begin{equation}\label{Ric:SLQ}\left\{\begin{aligned}
& \dot\cP + \cP\cA + \cA^\top\cP + \cC^\top\cP\cC + \cQ\\
& \hp{\dot\cP} -(\cP\cB + \cC^\top\cP\cD + \cS^\top)(\cR+\cD^\top\cP\cD)^{-1}
                (\cB^\top\cP + \cD^\top\cP\cC + \cS)=0, \\
& \cP(T) = \cG
\end{aligned}\right.\end{equation}
admits a unique solution $\cP\in C([0,T];\dbS^n)$ such that
\begin{equation*}
   \cR+\cD^\top\cP\cD \gg 0.
\end{equation*}
Moreover, for each initial state $x$, a unique open-loop optimal control exists and is given by
the following closed-loop form:
$$ v^* = -(\cR+\cD^\top\cP\cD)^{-1}(\cB^\top\cP + \cD^\top\cP\cC + \cS)X, $$
and the value at $x$ is given by
$$ \cV(x) = \lan \cP(0)x,x\ran, \q\forall x\in\dbR^n. $$
\end{lemma}

We have the following corollary to \autoref{lmm:SLQ-main}.

\begin{corollary}\label{crllry:standard-cndtn}
Suppose that
\begin{equation}\label{FSLQ:standard-condition}
\cG\ges0, \q \cR\gg0, \q \cQ-\cS^\top\cR^{-1}\cS\ges0.
\end{equation}
Then \rf{FSLQ:cJ-uni-tu} holds for some constant $\a>0$, and the solution of \rf{Ric:SLQ} satisfies
$$\cP(t)\ges0, \q\forall t\in[0,T]. $$
If, in addition to \rf{FSLQ:standard-condition}, $\cG>0$, then $\cP(t)>0$ for all $t\in[0,T]$.
\end{corollary}

\begin{proof}
Take an arbitrary initial state $x$ and an arbitrary control $v\in\sU$,
and let $\cX^v_x$ denote the solution of \rf{state:SLQ} corresponding to $x$ and $v$.
By Assumption \rf{FSLQ:standard-condition}, we have
\begin{align}\label{stdrd-cndtn:eqn1}
\cJ(x;v)
&= \dbE\lan\cG\cX^v_x(T),\cX^v_x(T)\ran + \dbE\int_0^T\Big\{\blan[\cQ(t)-\cS^\top(t)\cR^{-1}(t)\cS(t)]\cX^v_x(t),\cX^v_x(t)\bran \nn\\
&~\hp{=} +\big|\cR^{1\over2}(t)[v(t)+\cR^{-1}(t)\cS(t)\cX^v_x(t)]\big|^2\Big\}dt \nn\\
&\ges \dbE\lan\cG\cX^v_x(T),\cX^v_x(T)\ran + \dbE\int_0^T\big|\cR^{1\over2}(t)[v(t)+\cR^{-1}(t)\cS(t)\cX^v_x(t)]\big|^2 dt.
\end{align}
Since $\cG\ges0$ and $\cR(t)\ges \d I$ $\ae~t\in[0,T]$ for some constant $\d>0$, \rf{stdrd-cndtn:eqn1} implies
\begin{align}\label{SLQ:J>dEsth}
\cJ(x;v)  \ges \d\,\dbE\int_0^T\big|v(t)+\cR^{-1}(t)\cS(t)\cX^v_x(t)\big|^2dt \ges0.
\end{align}
For $x=0$, we can define a linear operator $\mathfrak{L}:\sU\to\sU$ by
$$ \mathfrak{L}v = v + \cR^{-1}\cS\cX^v_0.$$
It is easy to see that $\mathfrak{L}$ is bounded and bijective, with inverse $\mathfrak{L}^{-1}$ given by
$$ \mathfrak{L}^{-1} v = v - \cR^{-1}\cS\ti \cX^v_0,$$
where $\ti\cX^v_0$ is the solution of
$$\left\{\begin{aligned}
d\ti\cX^v_0(t) &= \big[(\cA-\cB\cR^{-1}\cS)\ti\cX^v_0+\cB v\big]dt \\
                &\hp{=\ } +\big[(\cC-\cD\cR^{-1}\cS)\ti\cX^v_0+\cD v\big]dW(t),\q t\in[0,T], \\
 \ti\cX^v_0(0) &=0.
\end{aligned}\right.$$
By the bounded inverse theorem, $\mathfrak{L}^{-1}$ is also bounded with $\|\mathfrak{L}^{-1}\|>0$. Thus,
\begin{align*}
\cJ(0;v)
&\ges \d\,\dbE\int_0^T\big|v(t)+\cR^{-1}(t)\cS(t)\cX^v_0(t)\big|^2dt = \d\,\dbE\int_0^T|(\mathfrak{L}v)(t)|^2dt\\
&\ges {\d\over\|\mathfrak{L}^{-1}\|^2}\,\dbE\int_0^T|v(t)|^2dt.
\end{align*}
This shows that \rf{FSLQ:cJ-uni-tu} holds with $\a={\d\over\|\mathfrak{L}^{-1}\|^2}$.
From \rf{SLQ:J>dEsth} and \autoref{lmm:SLQ-main} we obtain
$$ \lan \cP(0)x,x\ran = \inf_{v\in\sU}\cJ(x;v) \ges0, \q\forall x\in\dbR^n, $$
which implies $\cP(0)\ges0$. To see that $\cP(0)>0$ when $\cG>0$, we assume to the contrary that
$\lan \cP(0)x,x\ran=0$ for some $x\ne0$.
Let $\bar v$ be the optimal control for $x$. Then we have from \rf{stdrd-cndtn:eqn1} that
\begin{align*}
0 = \cJ(x;\bar v) \ges \dbE\lan\cG\cX^{\bar v}_x(T),\cX^{\bar v}_x(T)\ran
+ \dbE\int_0^T\big|\cR^{1\over2}(t)[\bar v(t)+\cR^{-1}(t)\cS(t)\cX^{\bar v}_x(t)]\big|^2 dt.
\end{align*}
Since $\cG>0$ and $\cR\gg0$, the above implies
$$ \cX^{\bar v}_x(T)=0, \q  \bar v(t)=-\cR^{-1}(t)\cS(t)\cX^{\bar v}_x(t). $$
Therefore, $\cX^{\bar v}_x$ satisfies the following equation:
$$\left\{\begin{aligned}
d\cX(t) &= [\cA(t)-\cB(t)\cR^{-1}(t)\cS(t)]\cX(t)dt + [\cC(t)-\cD(t)\cR^{-1}(t)\cS(t)]\cX(t)dW(t),  \\
 \cX(0) &= x, \q \cX(T)=0,
\end{aligned}\right.$$
which is impossible because $x\ne0$.
Finally, by considering Problem (FSLQ) over the interval $[t,T]$ and repeating the preceding argument,
we obtain $\cP(t)\ges0$ (respectively, $\cP(t)>0$ if $\cG>0$).
\end{proof}

\section{A study from the Hilbert space point of view}\label{sec:Hilbert}

Observe that $L^2_{\cF_T}(\Om;\dbR^n)$, $\sU$, and $L_\dbF^2(0,T;\dbR^n)$ are Hilbert spaces equipped with
their usual $L^2$-inner products, and that $L_\dbF^2(\Om;C([0,T];\dbR^n))\subseteq L_\dbF^2(0,T;\dbR^n)$.
For a terminal state $\xi\in L^2_{\cF_T}(\Om;\dbR^n)$ and a control $u\in\sU$, we denote by
$(Y_\xi^u,Z_\xi^u)$ the adapted solution to the state equation \rf{state}.
By the linearity of the state equation,
$$ Y_\xi^u = Y_\xi^0 + Y_0^u, \q Z_\xi^u = Z_\xi^0 + Z_0^u. $$
Note that $(Y_\xi^0,Z_\xi^0)$ linearly depends on $\xi$ and $(Y_0^u,Z_0^u)$ linearly depends on $u$.
Thus, \autoref{lmm:Sun2019} implies that the following are all bounded linear operators:
\begin{alignat*}{3}
&\cL_0: L^2_{\cF_T}(\Om;\dbR^n)\to \dbR^n;               \qq&& \cL_0\xi \deq Y_\xi^0(0), \\
&\cL_1: L^2_{\cF_T}(\Om;\dbR^n)\to L_\dbF^2(0,T;\dbR^n); \qq&& \cL_1\xi \deq Y_\xi^0, \\
&\cL_2: L^2_{\cF_T}(\Om;\dbR^n)\to L_\dbF^2(0,T;\dbR^n); \qq&& \cL_2\xi \deq Z_\xi^0, \\
&\cK_0: \sU\to \dbR^n;                                   \q&&  \cK_0u   \deq Y_0^u(0),\\
&\cK_1: \sU\to L_\dbF^2(0,T;\dbR^n);                     \qq&& \cK_1u   \deq Y_0^u, \\
&\cK_2: \sU\to L_\dbF^2(0,T;\dbR^n);                     \qq&& \cK_2u   \deq Z_0^u.
\end{alignat*}
Let us denote by $\cA^*$ the adjoint operator of a bounded linear operator $\cA$ between Banach spaces,
and denote the inner product of two elements $\phi$ and $\f$  in an $L^2$ space by $[\![\phi,\f]\!]$. Set
\begin{align*}
  M = \begin{pmatrix}  Q & S_1^\top & S_2^\top\\
                    S_1 & R_{11}   & R_{12}  \\
                    S_2 & R_{21}   & R_{22}  \end{pmatrix}, \q
\cL = \begin{pmatrix}\cL_1 \\ \cL_2 \\ 0\end{pmatrix}, \q
\cK = \begin{pmatrix}\cK_1 \\ \cK_2 \\ \cI\end{pmatrix},
\end{align*}
where $\cI$ is the identity operator. In terms of the above notation, the cost functional \rf{cost} can be written as follows:
\begin{align}\label{J-rep}
J(\xi;u)
&= [\![G(\cL_0\xi+\cK_0u),\cL_0\xi+\cK_0u]\!] + [\![M(\cL\xi+\cK u),\cL\xi+\cK u]\!] \nn\\
&= [\![(\cK_0^*G\cK_0 + \cK^*M\cK)u,u]\!]
   +2[\![(\cK_0^*G\cL_0 + \cK^*M\cL)\xi,u]\!] \nn\\
&~\hp{=} +[\![(\cL_0^*G\cL_0 + \cL^*M\cL)\xi,\xi]\!]  \nn\\
&\equiv [\![\cA u,u]\!] + 2[\![\cB\xi,u]\!] + [\![\cC\xi,\xi]\!].
\end{align}
Clearly,
\begin{alignat*}{3}
\cA &= \cK_0^*G\cK_0 + \cK^*M\cK &&: ~ \sU \to \sU, \\
\cC &= \cL_0^*G\cL_0 + \cL^*M\cL &&: ~ L^2_{\cF_T}(\Om;\dbR^n)\to L^2_{\cF_T}(\Om;\dbR^n)
\end{alignat*}
are bounded linear self-adjoint operators, and
$$ \cB = \cK_0^*G\cL_0 + \cK^*M\cL: ~ L^2_{\cF_T}(\Om;\dbR^n)\to \sU $$
is a bounded linear operator.

\ms

From the representation \rf{J-rep}, we obtain the following characterization of optimal controls.

\begin{theorem}\label{thm:iff-Hilbert-method}
Let \ref{ass:A1}--\ref{ass:A2} hold. For a given terminal state $\xi\in L^2_{\cF_T}(\Om;\dbR^n)$,
a control $u^*\in\sU$ is optimal if and only if
\begin{enumerate}[(i)]
\item $\cA\ges 0$ (i.e., $\cA$ is a positive operator), and
\item $\cA u^*+\cB\xi =0$.
\end{enumerate}
\end{theorem}

\begin{proof}
For every $u\in\sU$ and $\l\in\dbR$, we have
\begin{align*}
J(\xi;u^*+\l u)
  &= \Lan \cA(u^*+\l u),u^*+\l u \Ran + 2\Lan \cB\xi,u^*+\l u \Ran + \Lan \cC\xi,\xi \Ran\\
  &= J(\xi;u^*) + \l^2\Lan\cA u,u\Ran + 2\l\Lan\cA u^*+ \cB\xi,u\Ran,
\end{align*}
from which it follows that $u^*$ is optimal for $\xi$ if and only if
\begin{align}\label{Hilbert:iff-1}
\l^2\Lan\cA u,u\Ran + 2\l\Lan\cA u^*+ \cB\xi,u\Ran \ges0, \q\forall\l\in\dbR,~\forall u\in\sU.
\end{align}
Clearly, (i) and (ii) imply \rf{Hilbert:iff-1}.
Conversely, if \rf{Hilbert:iff-1} holds, taking $\l=\pm1$ in \rf{Hilbert:iff-1}
and then adding the resulting inequalities, we obtain $\cA\ges 0$.
Dividing both sides of the inequality in \rf{Hilbert:iff-1} by $\l>0$ and then sending $\l\to0$ gives
$$ \Lan\cA u^*+ \cB\xi,u\Ran \ges0, \q\forall u\in\sU. $$
Dividing both sides of the inequality in \rf{Hilbert:iff-1} by $\l<0$ and then sending $\l\to0$ gives
$$ \Lan\cA u^*+ \cB\xi,u\Ran \les0, \q\forall u\in\sU. $$
Therefore, $\Lan\cA u^*+ \cB\xi,u\Ran=0$ for all $u\in\sU$ and thereby (ii) holds.
\end{proof}

We see from \autoref{thm:iff-Hilbert-method} that the positivity condition $\cA\ges 0$,
which by \rf{J-rep} is equivalent to
$$ J(0;u) \ges 0, \q\forall u\in\sU, $$
is necessary for the existence of an optimal control.
On the other hand, if $\cA$ is uniformly positive, i.e., there exists a constant $\d>0$ such that
$$ \Lan \cA u,u \Ran = J(0;u) \ges \d\,\dbE\int_0^T|u(t)|^2dt, \q\forall u\in\sU, $$
then by \autoref{thm:iff-Hilbert-method}, for each $\xi\in L^2_{\cF_T}(\Om;\dbR^n)$,
an optimal control $u^*$ uniquely exists and is given by
$$ u^* = -\cA^{-1}\cB\xi. $$
When the necessity condition $\cA\ges 0$ holds but it is not clear if $\cA$ is uniformly positive,
we can define, for each $\e>0$, a new cost functional $J_\e(\xi;u)$ by
\begin{align*}
J_\e(\xi;u)
  &= J(\xi;u) + \e\,\dbE\int_0^T|u(t)|^2dt \\
  &= \Lan (\cA+\e\cI)u,u \Ran + 2\Lan \cB\xi,u \Ran + \Lan \cC\xi,\xi \Ran.
\end{align*}
For the new cost functional, the (unique) optimal control $u^*_\e$ for $\xi$ exists and is given by
\begin{align*}
u^*_\e = -(\cA+\e\cI)^{-1}\cB\xi.
\end{align*}
In terms of the family $\{u^*_\e\}_{\e>0}$, we now provide a sufficient and necessary condition
for the existence of optimal controls.

\begin{theorem}\label{thm:A+eI}
Let \ref{ass:A1}--\ref{ass:A2} hold and assume that the necessity condition $\cA\ges 0$ holds.
Then an optimal control exists for a given terminal state $\xi\in L^2_{\cF_T}(\Om;\dbR^n)$
if and only if one of the following conditions holds:
\begin{enumerate}[(i)]
\item the family $\{u_\e^*\}_{\e>0}$ is bounded in the Hilbert space $\sU$;
\item $u_\e^*$ converges weakly in $\sU$ as $\e\to0$;
\item $u_\e^*$ converges strongly in $\sU$ as $\e\to0$.
\end{enumerate}
Whenever (i), (ii), or (iii) is satisfied, the strong (weak) limit $u^*=\lim_{\e\to0}u^*_\e$
is an optimal control for $\xi$.
\end{theorem}

\begin{proof}
Suppose that $u^*\in\sU$ is optimal for $\xi$. Then by \autoref{thm:iff-Hilbert-method},
$$ \cA u^*+\cB\xi =0. $$
Write $u^*$ as $u^*=u+v$ with $u\in\ker\cA$ and $v\in\overline{\im\cA}$, where $\ker\cA$ and $\overline{\im\cA}$
are the kernel and the closure of the image of $\cA$, respectively. Then
$$ \cA v+\cB\xi =0, $$
and thereby $v$ is also an optimal control for $\xi$.
For a fixed but arbitrary $\d>0$, since $v\in\overline{\im\cA}$, we can find a $w\in\cU$ such that
$\|\cA w-v\|\les\d$. Then using the fact
$$ -\cB\xi = \cA v, \q \|(\cA+\e\cI)^{-1}\| \les \e^{-1}, \q \|(\cA+\e\cI)^{-1}\cA\| \les 1, $$
we have for any $\e>0$ that
\begin{align*}
\|u_\e^*-v\|
  &= \|-(\cA+\e\cI)^{-1}\cB\xi-v\| = \|(\cA+\e\cI)^{-1}\cA v-v\| = \e\|(\cA+\e\cI)^{-1}v\| \\
  &\les \e\|(\cA+\e\cI)^{-1}(v-\cA w)\| + \e\|(\cA+\e\cI)^{-1}\cA w\|  \\
  &\les \d + \e\|w\|.
\end{align*}
Letting $\e\to0$ and noting that $\d>0$ is arbitrary, we conclude that $u_\e^*$ converges strongly
to the optimal control $v$ as $\e\to0$.

\ms

It is clear that (iii) $\Ra$ (ii) $\Ra$ (i).
So it remains to show that (i) implies the existence of an optimal control for $\xi$.
When (i) holds, we can select a sequence $\{\e_n\}_{n=1}^\i$ with $\e_n$ decreasing to 0
(as $n\to\i$) such that $u^*_{\e_n}$ converges weakly to some $v^*\in\sU$.
By Mazur's lemma, for each integer $k\ges1$, there exists finite many positive numbers
$\a_{k1},\ldots,\a_{kN_k}$ with $\a_{k1}+\cdots+\a_{kN_k}=1$ such that
$v_k\deq \sum_{j=1}^{N_k}\a_{kj}u^*_{\e_{k+j}}$ converges strongly to $v^*$ as $k\to\i$.
Then
\begin{align}\label{Hilbert:J(v*)}
J(\xi;v^*)
&= \lim_{k\to\i}\(\Lan\cA v_k,v_k\Ran + 2\Lan\cB\xi,v_k\Ran + \Lan\cC\xi,\xi\Ran\) \nn\\
&\les \liminf_{k\to\i}\sum_{j=1}^{N_k}\a_{kj}\(\Lan\cA u^*_{\e_{k+j}},u^*_{\e_{k+j}}\Ran
      +2\Lan\cB\xi,u^*_{\e_{k+j}}\Ran + \Lan\cC\xi,\xi\Ran\) \nn\\
&\les \liminf_{k\to\i}\sum_{j=1}^{N_k}\a_{kj}\(\Lan(\cA+\e_{k+j}\cI) u^*_{\e_{k+j}},u^*_{\e_{k+j}}\Ran
      +2\Lan\cB\xi,u^*_{\e_{k+j}}\Ran + \Lan\cC\xi,\xi\Ran\) \nn\\
&= \liminf_{k\to\i}\sum_{j=1}^{N_k}\a_{kj}J_{\e_{k+j}}(\xi;u^*_{\e_{k+j}}).
\end{align}
Since $\lim_{k\to\i}\sum_{j=1}^{N_k}\a_{kj}\e_{k+j}=0$ and for any $u\in\sU$,
\begin{align*}
J_{\e_{k+j}}(\xi;u^*_{\e_{k+j}})\les J_{\e_{k+j}}(\xi;u) = J(\xi;u) + \e_{k+j}\|u\|^2,
\end{align*}
we conclude from \rf{Hilbert:J(v*)} that
$$J(\xi;v^*)\les J(\xi;u), \q\forall u\in\sU. $$
This shows that $v^*$ is an optimal control for $\xi$ and hence completes the proof.
\end{proof}

\section{A characterization of optimal controls in terms of FBSDEs}\label{sec:FBSDE}

In the previous section, \autoref{thm:iff-Hilbert-method} provides a characterization of the optimal control
using the operators in \rf{J-rep}. We now present an alternative characterization in terms of forward-backward
stochastic differential equations, which is more convenient for the verification of optimal controls.
This result will be used in \autoref{sec:construction} to prove the control constructed there is optimal.

\begin{theorem}\label{thm:optimality}
Let \ref{ass:A1}--\ref{ass:A2} hold and let the terminal state $\xi\in L^2_{\cF_T}(\Om;\dbR^n)$ be given.
A control $u^*\in\sU$ is optimal for $\xi$ if and only if the following conditions hold:
\begin{enumerate}[(i)]
\item $J(0;u)\ges0$ for all $u\in\sU$.
\item The adapted solution $(X^*,Y^*,Z^*)$ to the decoupled FBSDE
      \begin{equation}\label{FBSDE}
      \left\{\begin{aligned}
        dX^*(t) &= (-A^\top X^* +QY^* +S_1^\top Z^* +S_2^\top u^*)dt \\
                &~\hp{=} +(-C^\top X^* +S_1Y^* +R_{11}Z^* +R_{12}u^*)dW, \\
        dY^*(t) &= (AY^* +Bu^* +CZ^*)dt + Z^*dW, \\
         X^*(0) &= GY^*(0), \q Y^*(T) = \xi,
      \end{aligned}\right.
      \end{equation}
      satisfies
      \begin{align}\label{DJ=0}
        S_2Y^* + R_{21}Z^* - B^\top X^*  + R_{22}u^*  =0.
      \end{align}
\end{enumerate}
\end{theorem}

\begin{proof}
First we note that $u^*\in\sU$ is optimal for $\xi$ if and only if
\begin{align}\label{J-J*>0}
J(\xi;u^*+\e u) - J(\xi;u^*)\ges0, \q\forall u\in\sU,~\forall \e\in\dbR.
\end{align}
Let $u\in\sU$ and $\e\in\dbR$ be fixed but arbitrary. Denote by $(Y,Z)$ the solution of
$$\left\{\begin{aligned}
dY(t) &= (AY +Bu +CZ)dt + ZdW,\\
 Y(T) &= 0,
\end{aligned}\right.$$
and by $(Y^\e,Z^\e)$ the solution of
$$\left\{\begin{aligned}
dY^\e(t) &= [AY^\e +B(u^*+\e u) + CZ^\e]dt + Z^\e dW, \\
 Y^\e(T) &=\xi.
\end{aligned}\right.$$
Clearly, $(Y^\e,Z^\e)=(Y^*+\e Y,Z^*+\e Z)$. Thus,
\begin{align}\label{J-J*}
& J(\xi;u^*+\e u) - J(\xi;u^*) \nn\\
&\q= 2\e\dbE\Bigg[\lan GY^*(0),Y(0)\ran
    +\int_0^T\Blan\begin{pmatrix}  Q & \!\!S_1^\top & \!\!S_2^\top \\
                                 S_1 & \!\!R_{11}   & \!\!R_{12}   \\
                                 S_2 & \!\!R_{21}   & \!\!R_{22}   \end{pmatrix}\!
    \begin{pmatrix}Y^* \\ Z^* \\ u^*\end{pmatrix}\!,
    \begin{pmatrix}Y   \\ Z   \\ u   \end{pmatrix}\Bran dt\Bigg] \nn\\
&\q~\hp{=} +\e^2\dbE\Bigg[\lan GY(0),Y(0)\ran
    +\int_0^T\Blan
    \begin{pmatrix}  Q & \!\!S_1^\top & \!\!S_2^\top\\
                   S_1 & \!\!R_{11}   & \!\!R_{12}  \\
                   S_2 & \!\!R_{21}   & \!\!R_{22}  \end{pmatrix}\!
    \begin{pmatrix}Y \\ Z \\ u\end{pmatrix}\!,
    \begin{pmatrix}Y \\ Z \\ u\end{pmatrix}\Bran dt\Bigg] \nn\\
&\q= 2\e\dbE\Bigg[\lan GY^*(0),Y(0)\ran +\int_0^T\Blan
     \begin{pmatrix}  Q & \!\!S_1^\top& \!\!S_2^\top\\
                    S_1 & \!\!R_{11}  & \!\!R_{12}  \\
                    S_2 & \!\!R_{21}  & \!\!R_{22}  \end{pmatrix}\!
     \begin{pmatrix}Y^*\\ Z^*\\ u^*\end{pmatrix}\!,
     \begin{pmatrix}Y  \\ Z  \\ u  \end{pmatrix}\Bran dt\Bigg] + \e^2J(0;u).
\end{align}
Using integration by parts we obtain
\begin{align*}
-\lan GY^*(0),Y(0)\ran
&= \dbE\int_0^T\[\lan QY^* + S_1^\top Z^* + S_2^\top u^*,Y\ran \\
&~\hp{=} +\lan S_1Y^* +R_{11}Z^* +R_{12}u^*,Z\ran +\lan B^\top X^*,u\ran\]dt,
\end{align*}
which, substituted in \rf{J-J*}, yields
\begin{align*}
J(\xi;u^*+\e u) - J(\xi;u^*)
&= \e^2J(0;u) + 2\e\dbE\int_0^T\lan S_2Y^* +R_{21}Z^* -B^\top X^* +R_{22}u^*,u\ran dt.
\end{align*}
From the above, it is not difficult to see that \rf{J-J*>0} holds if and only if \rf{DJ=0}
holds and $J(0;u)\ges0$ for all $u\in\sU$.
\end{proof}

\section{Connections with forward stochastic LQ optimal control problems}\label{sec:connection}

It is not easy to decide whether Problem (BSLQ) admits an optimal control (let alone constructing one)
for a given terminal state when the operator $\cA$ is merely positive.
However, if the case that $\cA$ is uniformly positive can be solved, then we are at least able to construct
a minimizing sequence for Problem (BSLQ), by using \autoref{thm:A+eI}.
To solve the uniform positivity case, we investigate in this section the connection between Problem (BSLQ)
and the forward stochastic LQ optimal control problem (FSLQ problem, for short) under
the uniform positivity condition:
\begin{taggedassumption}{(A3)}\label{ass:A3}
There exists a constant $\d>0$ such that
\begin{align}\label{Uni-positive}
\Lan \cA u,u \Ran = J(0;u) \ges \d\,\dbE\int_0^T|u(t)|^2dt, \q\forall u\in\sU.
\end{align}
\end{taggedassumption}

Let us consider the controlled linear forward stochastic differential equation
\begin{equation}\label{FLQ:state}
\left\{\begin{aligned}
dX(t) &= [A(t)X(t)+B(t)u(t)+C(t)v(t)]dt + v(t)dW(t), \q t\in[0,T],\\
 X(0) &= x,
\end{aligned}\right.
\end{equation}
and, for $\l>0$, the cost functional
\begin{align}\label{FLQ:cost}
\cJ_\l(x;u,v) \deq \dbE\bigg\{\l|X(T)|^2
+\!\int_0^T\!\Blan
          \begin{pmatrix}  Q(t) & \!\!S_1^\top(t) & \!\!S_2^\top(t)\\
                         S_1(t) & \!\!R_{11}(t)   & \!\!R_{12}(t)  \\
                         S_2(t) & \!\!R_{21}(t)   & \!\!R_{22}(t)  \end{pmatrix}\!\!
          \begin{pmatrix}X(t) \\ v(t) \\ u(t)\end{pmatrix}\!,\!
          \begin{pmatrix}X(t) \\ v(t) \\ u(t)\end{pmatrix}\Bran dt \bigg\}.
\end{align}
In the above, the control is the pair
$$(u,v) \in L_\dbF^2(0,T;\dbR^m)\times L_\dbF^2(0,T;\dbR^n) \equiv \sU\times\sV. $$
We impose the following FSLQ problem.

\begin{taggedthm}{Problem (FSLQ)$_\l$.}\label{Prob:FSLQ-l}
For a given initial state $x\in\dbR^n$, find a control $(u^*,v^*)\in\sU\times\sV$ such that
\begin{align*}
\cJ_\l(x;u^*,v^*) = \inf_{(u,v)\in\sU\times\sV}\cJ_\l(x;u,v)\equiv \cV_\l(x).
\end{align*}
\end{taggedthm}

We have the following result, which plays a basic role in the subsequent analysis.

\begin{theorem}\label{thm:FLQ-convex}
Let \ref{ass:A1}--\ref{ass:A2} hold. If \ref{ass:A3} holds, then there exist constants $\rho>0$ and $\l_0>0$
such that for $\l\ges\l_0$,
\begin{align}\label{cJ-tu}
\cJ_\l(0;u,v) \ges \rho\,\dbE\int_0^T\[|u(t)|^2+|v(t)|^2\]dt, \q\forall (u,v)\in\sU\times\sV.
\end{align}
If, in addition, $G=0$, then for $\l\ges\l_0$,
\begin{align}\label{cJ-tu2}
\cJ_\l(x;u,v) \ges \rho\,\dbE\int_0^T\[|u(t)|^2+|v(t)|^2\]dt, \q\forall (u,v)\in\sU\times\sV,~\forall x\in\dbR^n.
\end{align}
\end{theorem}

\begin{proof}
Fix $x\in\dbR^n$ and $(u,v)\in\sU\times\sV$, and let $X$ be the solution of \rf{FLQ:state}.
We see that the random variable $\xi\deq X(T)$ belongs to the space $L^2_{\cF_T}(\Om;\dbR^n)$.
Denote by $(Y_0,Z_0)$ and $(Y,Z)$ the adapted solutions to
\begin{equation*}
\left\{\begin{aligned}
dY_0(t) &= [A(t)Y_0(t) + C(t)Z_0(t)]dt + Z_0(t)dW(t), \\
 Y_0(T) &= \xi,
\end{aligned}\right.
\end{equation*}
and
\begin{equation*}
\left\{\begin{aligned}
dY(t) &= [A(t)Y(t) + B(t)u(t) + C(t)Z(t)]dt + Z(t)dW(t), \\
 Y(T) &= 0,
\end{aligned}\right.
\end{equation*}
respectively. Regarding $(X(t),v(t))$ as the variables for a BSDE with terminal $\xi$,
then we have by the uniqueness of an adapted solution that
$$ X(t) = Y(t)+Y_0(t), \q v(t) = Z(t)+Z_0(t); \q t\in[0,T]. $$
For simplicity, we introduce the following the notation:
\begin{align*}
M(t) = \begin{pmatrix} Q(t) & S_1^\top(t) & S_2^\top(t)\\
                     S_1(t) & R_{11}(t)   & R_{12}(t)  \\
                     S_2(t) & R_{21}(t)   & R_{22}(t)  \end{pmatrix}, \q
\a(t) = \begin{pmatrix}  Y(t) \\   Z(t) \\ u(t)\end{pmatrix}, \q
\b(t) = \begin{pmatrix}Y_0(t) \\ Z_0(t) \\ 0   \end{pmatrix}.
\end{align*}
Then we have
$$ J(0;u)= \dbE\bigg\{\lan GY(0),Y(0)\ran +\int_0^T\lan M(t)\a(t),\a(t)\ran dt\bigg\} $$
and hence
\begin{align}\label{cJ-guji}
\cJ_\l(x;u,v)
&= \dbE\bigg\{\l|X(T)|^2 + \int_0^T\lan M(t)[\a(t)+\b(t)],\a(t)+\b(t)\ran dt \bigg\} \nn\\
&= J(0;u) + \dbE\bigg\{\l|X(T)|^2 -\lan GY(0),Y(0)\ran \nn\\
&~\hp{=} +\int_0^T\lan M(t)\b(t),\b(t)\ran dt + 2\int_0^T\lan M(t)\a(t),\b(t)\ran dt \bigg\}.
\end{align}
Since the weighting matrices in the cost functional are bounded, we can chose a constant $K\ges1$
such that $|M(t)|\les K$ for a.e. $t\in[0,T]$. Thus, by the Cauchy--Schwarz inequality we have
\begin{align}\label{guji-1}
&\lt|\dbE\lt\{\int_0^T\lan M(t)\b(t),\b(t)\ran dt + 2\int_0^T\lan M(t)\a(t),\b(t)\ran dt \rt\}\rt|\nn\\
&\q\les K\lt\{\dbE\int_0^T|\b(t)|^2dt + 2\dbE\int_0^T|\a(t)||\b(t)| dt \rt\} \nn\\
&\q \les K\lt\{(\m+1)\,\dbE\int_0^T|\b(t)|^2dt + {1\over\m}\,\dbE\int_0^T|\a(t)|^2 dt\rt\},
\end{align}
where $\m>0$ is a constant to be chosen later.
If we choose $K\ges1$ large enough (still independent of $X(T)$ and $(u,v)$), then according to \autoref{lmm:Sun2019},
\begin{align}\label{guji-2}
\dbE\int_0^T|\a(t)|^2 dt
\les K\dbE\int_0^T|u(t)|^2dt, \q
\dbE\int_0^T|\b(t)|^2dt
\les K\dbE|X(T)|^2,
\end{align}
and for the case $x=0$, we have
\begin{align}\label{guji-3}
|\lan GY(0),Y(0)\ran| = |\lan GY_0(0),Y_0(0)\ran| \les K\dbE|X(T)|^2.
\end{align}
Combining \rf{guji-1} and \rf{guji-2}, we obtain from \rf{cJ-guji} that
\begin{align}\label{cJ:2020-11-8}
\cJ_\l(x;u,v)
&\ges J(0;u) -\lan GY(0),Y(0)\ran +[\l-K^2(\m+1)]\dbE|X(T)|^2 - {K^2\over\m}\,\dbE\int_0^T|u(t)|^2dt \nn\\
&\ges \bigg(\d \!-\! {K^2\over\m}\bigg)\dbE\!\int_0^T\!|u(t)|^2dt + [\l \!-\! K^2(\m\!+\!1)]\dbE|X(T)|^2 - \lan GY(0),Y(0)\ran.
\end{align}
From \rf{guji-2} we have
\begin{align*}
\dbE\int_0^T|v(t)|^2dt
&= \dbE\int_0^T|Z(t)+Z_0(t)|^2dt \les 2\lt[\dbE\int_0^T|Z(t)|^2dt+\dbE\int_0^T|Z_0(t)|^2dt\rt] \\
&\les 2\lt[\dbE\int_0^T|\a(t)|^2dt+\dbE\int_0^T|\b(t)|^2dt\rt] \\
&\les 2K\lt[\dbE|X(T)|^2+\dbE\int_0^T|u(t)|^2dt\rt],
\end{align*}
which implies that
\begin{align}\label{E|X(T)|>|u|+|v|}
\dbE|X(T)|^2 \ges {1\over2K}\,\dbE\int_0^T|v(t)|^2dt - \dbE\int_0^T|u(t)|^2dt.
\end{align}
Taking $\m={2K^2\over\d}$ and substituting \rf{E|X(T)|>|u|+|v|} into \rf{cJ:2020-11-8},
we see that when $\l\ges \l_0\deq {\d\over4}+K+K^2(\m+1)$,
\begin{align}\label{Jl(x;u,v);guji}
\cJ_\l(x;u,v)
&\ges {\d\over2}\,\dbE\int_0^T|u(t)|^2dt + \({\d\over4}+K\)\,\dbE|X(T)|^2 -\lan GY(0),Y(0)\ran \nn\\
&\ges {\d\over4}\,\dbE\int_0^T|u(t)|^2dt + {\d\over 8K}\,\dbE\int_0^T|v(t)|^2dt + \[K\dbE|X(T)|^2 -\lan GY(0),Y(0)\ran\] \nn\\
&\ges {\d\over8K}\,\dbE\int_0^T\[|u(t)|^2+|v(t)|^2\]dt + \[K\dbE|X(T)|^2 -\lan GY(0),Y(0)\ran\].
\end{align}
When $x=0$, by using \rf{guji-3} we further obtain
\begin{align*}
\cJ_\l(0;u,v) \ges {\d\over8K}\,\dbE\int_0^T\[|u(t)|^2+|v(t)|^2\]dt .
\end{align*}
If $G=0$, we obtain from \rf{Jl(x;u,v);guji} that
$$ \cJ_\l(x;u,v)\ges {\d\over8K}\,\dbE\int_0^T\[|u(t)|^2+|v(t)|^2\]dt $$
for all $(u,v)\in\sU\times\sV$ and all $x\in\dbR^n$.
\end{proof}

Combining \autoref{thm:FLQ-convex} and \autoref{lmm:SLQ-main}, we get the following corollaries.

\begin{corollary}\label{crllry:Vl(x)>0}
Under the assumptions of \autoref{thm:FLQ-convex}, Problem (FSLQ)$_\l$ is uniquely solvable for $\l\ges\l_0$.
If, in addition, $G=0$, then for $\l\ges\l_0$ the value function $\cV_\l$ satisfies
$$ \cV_\l(x)\ges0, \q\forall x\in\dbR^n. $$
\end{corollary}

\begin{proof}
The unique solvability of Problem (FSLQ)$_\l$ follows from \rf{cJ-tu} and \autoref{lmm:SLQ-main}.
When $G=0$, \rf{cJ-tu2} implies that $\cV_\l(x)\ges0$ for all $\l\ges\l_0$ and all $x\in\dbR^n$.
\end{proof}

\begin{corollary}\label{crllry:Pl-exist}
Under the assumptions of \autoref{thm:FLQ-convex}, for $\l\ges\l_0$ the Riccati equation
\begin{equation}\label{Ric:P}
\left\{\begin{aligned}
&\dot P_\l + P_\l A+A^\top P_\l + Q \\
&\hp{P_\l} -\begin{pmatrix}C^\top P_\l + S_1 \\
                           B^\top P_\l + S_2 \end{pmatrix}^{\!\top}\!
            \begin{pmatrix}R_{11}+P_\l & R_{12} \\
                           R_{21}      & R_{22} \end{pmatrix}^{-1}\!
            \begin{pmatrix}C^\top P_\l + S_1 \\
                           B^\top P_\l + S_2 \end{pmatrix} =0,\\
&P_\l(T)= \l I,
\end{aligned}\right.
\end{equation}
admits a unique solution $P_\l\in C([0,T];\dbS^n$) such that
\begin{equation}\label{Ric:R+DPD>0}
  \begin{pmatrix}R_{11}+P_\l & R_{12} \\ R_{21} & R_{22}\end{pmatrix}\gg0.
\end{equation}
Moreover, $\cV_\l(x) = \lan P_\l(0)x,x\ran$ for all $x\in\dbR^n$.
\end{corollary}

\begin{proof}
It follows directly from \rf{cJ-tu} and \autoref{lmm:SLQ-main}.
\end{proof}

We conclude this section with a remark on the weighting matrix $R_{22}$ of the control process.

\begin{remark}\label{remark:R22}
Clearly, \rf{Ric:R+DPD>0} implies that $R_{22}\gg0$. Thus, in order for the cost functional of
Problem (BSLQ) to be uniformly positive (or equivalently, in order for \ref{ass:A3} to hold),
the weighting matrix $R_{22}(t)$ must be positive definite uniformly in $t\in[0,T]$.
This is quite different from the forward stochastic LQ optimal control problem, in which the
uniform positivity of the weighting matrix for the control is neither sufficient nor necessary
for the uniform positivity of the cost functional.
In a similar manner, we can show that in order for the cost functional of Problem (BSLQ) to be
positive, the weighting matrix $R_{22}(t)$ must be positive definite for a.e. $t\in[0,T]$.
\end{remark}

\section{Construction of optimal controls}\label{sec:construction}

In this section we construct the optimal control of Problem (BSLQ) under the uniform positivity condition \ref{ass:A3}.
As mentioned in \autoref{sec:connection}, once the case of \ref{ass:A3} is solved, we can develop an $\e$-approximation
scheme that is asymptotically optimal for the general case, thanks to \autoref{thm:A+eI}.

\ms

First, we observe that the uniform positivity condition \ref{ass:A3} implies $R_{22}\gg0$ (\autoref{remark:R22}).
This enables us to simplify Problem (BSLQ) by assuming
\begin{align}\label{G=Q=R12=0}
G=0, \q Q(t)=0, \q R_{12}(t)=R_{21}^\top(t)=0; \q\forall t\in[0,T].
\end{align}
In fact, using the transformations
\begin{equation}\label{transfer}
\begin{aligned}
\sS_1 &= S_1-R_{12}R_{22}^{-1}S_2, \q & \sR_{11} &= R_{11}-R_{12}R_{22}^{-1}R_{21}, \\
\sC   &= C-BR_{22}^{-1}R_{21},     \q &        v &= u+R_{22}^{-1}R_{21}Z,
\end{aligned}
\end{equation}
the original Problem (BSLQ) is equivalent to the backward stochastic LQ optimal control problem with state equation
\begin{equation}\label{state:sY}
\left\{\begin{aligned}
dY(t) &= [A(t)Y(t)+B(t)v(t)+\sC(t)Z(t)]dt + Z(t)dW(t), \\
 Y(T) &= \xi,
\end{aligned}\right.
\end{equation}
and cost functional
\begin{align}\label{cost:sJ}
\sJ(\xi;v) &= \dbE\bigg\{\lan GY(0),Y(0)\ran + \!\int_0^T\!\Blan
          \begin{pmatrix}  Q(t) & \!\!\sS_1^\top(t) & \!\!S_2^\top(t)\\
                       \sS_1(t) & \!\!\sR_{11}(t)   & \!\!0          \\
                         S_2(t) & \!\!0             & \!\!R_{22}(t)  \end{pmatrix}\!\!
          \begin{pmatrix}Y(t) \\ Z(t) \\ v(t)\end{pmatrix}\!,\!
          \begin{pmatrix}Y(t) \\ Z(t) \\ v(t)\end{pmatrix}\Bran dt \bigg\}.
\end{align}
Furthermore, letting $H\in C([0,T];\dbS^n)$ be the unique solution to the linear ordinary differential equation (ODE, for short)
\begin{equation}\label{Ric:H(0)=G}
\left\{\begin{aligned}
  &\dot H(t) + H(t)A(t) + A(t)^\top H(t) + Q(t) =0, \q t\in[0,T],\\
  &H(0)= G,
\end{aligned}\right.
\end{equation}
and then applying the integration by parts formula to $t\mapsto\lan H(t)Y(t),Y(t)\ran$,
where $Y$ is the state process determined by \rf{state:sY}, we obtain
\begin{align*}
&\dbE\lan H(T)\xi,\xi\ran-\dbE\lan GY(0),Y(0)\ran \\
&\q= \dbE\int_0^T\[\lan(\dot H+HA+A^\top H)Y,Y\ran + 2\lan B^\top HY,v\ran + 2\lan\sC^\top HY,Z\ran + \lan HZ,Z\ran\]dt \\
&\q= \dbE\int_0^T\[-\lan QY,Y\ran + 2\lan B^\top HY,v\ran + 2\lan\sC^\top HY,Z\ran + \lan HZ,Z\ran\]dt \\
&\q= \dbE\int_0^T\Blan\begin{pmatrix} -Q & H\sC & HB \\
                              \sC^\top H & H    & 0  \\
                                B^\top H & 0    & 0  \end{pmatrix}\!
                      \begin{pmatrix}Y \\ Z \\ v\end{pmatrix},
                      \begin{pmatrix}Y \\ Z \\ v\end{pmatrix}\Bran dt.
\end{align*}
Substituting for $\dbE\lan GY(0),Y(0)\ran$ in the cost functional \rf{cost:sJ} yields
\begin{align*}
\sJ(\xi;v) &=\dbE\int_0^T\!\Blan
          \begin{pmatrix}  0 & \!\!(S_{^1}^{_H})^\top & \!\!(S_{^2}^{_H})^\top\\
                 S_{^1}^{_H} & \!\!    R_{^{11}}^{_H} & \!\!0  \\
                 S_{^2}^{_H} & \!\!                 0 & \!\!R_{22}  \end{pmatrix}\!\!
          \begin{pmatrix}Y \\ Z \\ v\end{pmatrix}\!,\!
          \begin{pmatrix}Y \\ Z \\ v\end{pmatrix}\Bran dt -\dbE\lan H(T)\xi,\xi\ran,
\end{align*}
where
\begin{align}\label{SH+RH}
 S_{^1}^{_H} = \sS_1+\sC^\top H, \q S_{^2}^{_H} = S_2+B^\top H, \q R_{^{11}}^{_H} = \sR_{11}+H.
\end{align}
Thus, for a given terminal state $\xi$, minimizing $J(\xi;u)$ subject to \rf{state}
is equivalent to minimizing the cost functional
\begin{align}\label{cost-H}
J^{_H}(\xi;v) &= \dbE\int_0^T\!\Blan
          \begin{pmatrix}  0 & \!\!(S_{^1}^{_H})^\top & \!\!(S_{^2}^{_H})^\top\\
                 S_{^1}^{_H} & \!\!    R_{^{11}}^{_H} & \!\!0 \\
                 S_{^2}^{_H} & \!\!                 0 & \!\!R_{22} \end{pmatrix}\!\!
          \begin{pmatrix}Y \\ Z \\ v\end{pmatrix}\!,\!
          \begin{pmatrix}Y \\ Z \\ v\end{pmatrix}\Bran dt,
\end{align}
subject to the state equation \rf{state:sY}.
Therefore, in the rest of this section we may assume without loss of generality that \rf{G=Q=R12=0} holds.
The general case will be discussed in \autoref{sec:conclusion}.

\ms

Observe that in the case of \rf{G=Q=R12=0}, the Riccati equation \rf{Ric:P} becomes
\begin{equation}\label{Ric:P(Q=0)}
\left\{\begin{aligned}
&\dot P_\l + P_\l A+A^\top P_\l
           - \begin{pmatrix}C^\top P_\l + S_1 \\
                            B^\top P_\l + S_2 \end{pmatrix}^{\!\top}\!
             \begin{pmatrix}R_{11}+P_\l & 0 \\
                                      0 & R_{22} \end{pmatrix}^{-1}\!
             \begin{pmatrix}C^\top P_\l + S_1 \\
                            B^\top P_\l + S_2 \end{pmatrix} =0, \\
&P_\l(T)= \l I.
\end{aligned}\right.
\end{equation}

\begin{proposition}\label{prop:Pl-property}
Let \ref{ass:A1}--\ref{ass:A3} and \rf{G=Q=R12=0} hold. Then for $\l\ges\l_0$,
the solution of \rf{Ric:P(Q=0)} satisfies
\begin{equation}\label{Ric:P(t)>0}
P_\l(t)\ges0, \q\forall t\in[0,T].
\end{equation}
Moreover, for every $\l_2>\l_1\ges\l_0$, we have
\begin{equation}\label{Pl:increase}
  P_{\l_2}(t) > P_{\l_1}(t), \q\forall t\in[0,T].
\end{equation}
\end{proposition}

\begin{proof}
Consider Problem (FSLQ)$_\l$ for $\l\ges\l_0$. Since $G=0$, we see from \autoref{crllry:Vl(x)>0} and \autoref{crllry:Pl-exist} that
$$ \lan P_\l(0)x,x\ran = \cV_\l(x) \ges 0, \q\forall x\in\dbR^n, $$
and hence $P_\l(0)\ges0$. With the notation
\begin{align*}
 \cQ_\l = \begin{pmatrix}C^\top P_\l + S_1 \\ B^\top P_\l + S_2 \end{pmatrix}^{\top}
          \begin{pmatrix}R_{11}+P_\l & 0 \\ 0 & R_{22} \end{pmatrix}^{-1}
          \begin{pmatrix}C^\top P_\l + S_1 \\ B^\top P_\l + S_2 \end{pmatrix}
\end{align*}
and with $\F$ denoting the solution to the matrix ODE
\begin{equation*}
\left\{\begin{aligned}
\dot\F(t) &= A(t)\F(t), \q t\in[0,T],\\
    \F(0) &= I_n,
\end{aligned}\right.
\end{equation*}
we can rewrite \rf{Ric:P(Q=0)} in the integral form
$$ P_\l(t) = \big[\F^{-1}(t)\big]^\top\lt[P_\l(0) + \int_0^t \F(s)^\top\cQ_\l(s)\F(s) ds\rt]\F^{-1}(t), \q t\in[0,T].$$
This implies \rf{Ric:P(t)>0} because $P_\l(0)\ges0$ and $\cQ_\l(t)\ges0$ a.e. by \rf{Ric:R+DPD>0}.
To prove \rf{Pl:increase}, let us consider $\cP(t)= P_{\l_2}(t)-P_{\l_1}(t)$, which satisfies the following equation:
\begin{equation}\label{P2-P1}
\left\{\begin{aligned}
&\dot\cP + \cP A+A^\top\cP + \cQ - (\cP\cB+\cS^\top) (\cR+\cD^\top \cP\cD)^{-1} (\cB^\top\cP+\cS)=0, \\
&\cP(T)= (\l_2-\l_1)I,
\end{aligned}\right.
\end{equation}
where we have employed the notation
\begin{eqnarray*}
& \cB = (C, B), \q \cD = (I,0), \q
  \cR = \begin{pmatrix}R_{11}+P_{\l_1} & 0 \\ 0 & R_{22} \end{pmatrix}, \q
  \cS = \begin{pmatrix}C^\top P_{\l_1} + S_1 \\ B^\top P_{\l_1} + S_2 \end{pmatrix}, \\
& \cQ = \begin{pmatrix}C^\top P_{\l_1} + S_1 \\ B^\top P_{\l_1} + S_2 \end{pmatrix}^{\top}
        \begin{pmatrix}R_{11}+P_{\l_1} & 0 \\ 0 & R_{22} \end{pmatrix}^{-1}
        \begin{pmatrix}C^\top P_{\l_1} + S_1 \\ B^\top P_{\l_1} + S_2 \end{pmatrix}=\cS^{\top}\cR^{-1}\cS.
\end{eqnarray*}
Clearly, the matrices $\cG=(\l_2-\l_1)I>0$, $\cQ$, $\cS$, and $\cR$ satisfies the condition \rf{FSLQ:standard-condition},
so \autoref{crllry:standard-cndtn} implies \rf{Pl:increase}.
\end{proof}

For notational convenience we write for an $\dbS^n$-valued function $\Si:[0,T]\to\dbS^n$,
\begin{align*}
& \cB(t,\Si(t)) = B(t)+\Si(t)S_2(t)^\top, \\
& \cC(t,\Si(t)) = C(t)+\Si(t)S_1(t)^\top, \\
& \cR(t,\Si(t)) = I + \Si(t)R_{11}(t).
\end{align*}
When there is no risk for confusion we will frequently suppress the argument $t$ from
our notation and write $\cB(t,\Si(t))$, $\cC(t,\Si(t))$, and $\cR(t,\Si(t))$ as $\cB(\Si)$,
$\cC(\Si)$, and $\cR(\Si)$, respectively. In order to construct the optimal control of
Problem (BSLQ), we now introduce the following Riccati equation:
\begin{equation}\label{Ric:Sigma(Q=0)}
\left\{\begin{aligned}
&\dot\Si(t) - A(t)\Si(t) - \Si(t)A(t)^\top + \cB(t,\Si(t))[R_{22}(t)]^{-1}\cB(t,\Si(t))^\top \\
&\hp{\dot\Si(t)} +\cC(t,\Si(t))[\cR(t,\Si(t))]^{-1}\Si(t)\cC(t,\Si(t))^\top=0, \q t\in[0,T],\\
&\Si(T)= 0.
\end{aligned}\right.
\end{equation}

\begin{theorem}
Let \ref{ass:A1}--\ref{ass:A3} and \rf{G=Q=R12=0} hold. Then the Riccati equation \rf{Ric:Sigma(Q=0)}
admits a unique positive semidefinite solution $\Si\in C([0,T];\dbS^n)$ such that $\cR(\Si)$ is invertible
a.e. on $[0,T]$ and $\cR(\Si)^{-1}\in L^\i(0,T;\dbR^n)$.
\end{theorem}

\begin{proof}
{\it Uniqueness.} Suppose that $\Si$ and $\Pi$ are two solutions of \rf{Ric:Sigma(Q=0)} satisfying the
properties stated in the theorem. Then $\D \deq \Si-\Pi$ satisfies $\D(T)=0$ and
\begin{align*}
\dot\D &= A\D + \D A^\top - \D S_2^\top R_{22}^{-1}\cB(\Si)^\top - \cB(\Pi)R_{22}^{-1}S_2\D \\
&~\hp{=} -\D S_1\cR(\Si)^{-1}\Si\cC(\Si)^\top -\cC(\Pi)\[\cR(\Si)^{-1}\Si\cC(\Si)^\top - \cR(\Pi)^{-1}\Pi\cC(\Pi)^\top\].
\end{align*}
Note that
\begin{align*}
& \cR(\Si)^{-1}\Si\cC(\Si)^\top - \cR(\Pi)^{-1}\Pi\cC(\Pi)^\top \\
&\q= -\,\cR(\Si)^{-1}\D R_{11}\cR(\Pi)^{-1}\Si\cC(\Si)^\top + \cR(\Pi)^{-1}\[\Si\cC(\Si)^\top - \Pi\cC(\Pi)^\top\]  \\
&\q= -\,\cR(\Si)^{-1}\D R_{11}\cR(\Pi)^{-1}\Si\cC(\Si)^\top + \cR(\Pi)^{-1}\[\D\cC(\Si)^\top + \Pi S_1\D\].
\end{align*}
It follows that
\begin{align*}
\dot\D(t) &= A\D + \D A^\top - \D S_2^\top R_{22}^{-1}\cB(\Si)^\top - \cB(\Pi)R_{22}^{-1}S_2\D -\D S_1\cR(\Si)^{-1}\Si\cC(\Si)^\top \\
&~\hp{=} +\cC(\Pi)\cR(\Si)^{-1}\D R_{11}\cR(\Pi)^{-1}\Si\cC(\Si)^\top -\cC(\Pi)\cR(\Pi)^{-1}\[\D\cC(\Si)^\top + \Pi S_1\D\] \\
&\equiv f(t,\D(t)).
\end{align*}
Noting that $\D(T)=0$ and $f(t,x)$ is Lipschitz-continuous in $x$,
we conclude by Gronwall's inequality that $\D(t)=0$ for all $t\in[0,T]$.

\ms

{\it Existence.} According to \autoref{prop:Pl-property}, for $\l>\l_0$, the solution $P_\l$ of \rf{Ric:P(Q=0)}
is positive definite on $[0,T]$. Thus we may define
$$ \Si_\l(t) \deq P_\l^{-1}(t), \q t\in[0,T]. $$
Again, by \autoref{prop:Pl-property}, for each fixed $t\in[0,T]$, $\Si_\l(t)$ is decreasing in $\l$ and bounded below by zero,
so the family $\{\Si_\l(t)\}_{\l>\l_0}$ is bounded uniformly in $t\in[0,T]$ and converges pointwise to some positive semidefinite
function $\Si:[0,T]\to\dbS^n$. Next we shall prove the following:
\begin{enumerate}[\indent(a)]
\item $\cR(t,\Si(t)) = I + \Si(t)R_{11}(t)$ is invertible for a.e. $t\in[0,T]$;
\item $\cR(\Si)^{-1}\in L^\i(0,T;\dbR^n)$; and
\item $\Si$ solves the equation \rf{Ric:Sigma(Q=0)}.
\end{enumerate}
For (a) and (b), we observe first that for $\l>\l_0$, $I + \Si_\l R_{11}$ is invertible
a.e. on $[0,T]$ since by \rf{Ric:R+DPD>0},
$$ P_\l(I + \Si_\l R_{11}) = P_\l +R_{11} \gg0. $$
Define $K=R_{11}+P_{\l_0}$ and $L_\l=P_\l-P_{\l_0}$. For every $\l>\l_0$,
$$ 0\les (K+L_\l)^{-1} = (R_{11}+P_\l)^{-1} \les (R_{11}+P_{\l_0})^{-1}, $$
from which we obtain
\begin{align*}
|(K+L_\l)^{-1}| \les |(R_{11}+P_{\l_0})^{-1}|, \q \forall \l>\l_0,
\end{align*}
and hence for every $x\in\dbR^n$,
\begin{align*}
\lan P_\l(R_{11}+P_\l)^{-2}P_\l x,x\ran
&= |(K+L_\l)^{-1}(L_\l+P_{\l_0})x|^2 \\
&\les 2|(K+L_\l)^{-1}L_\l x|^2 + 2|(K+L_\l)^{-1}P_{\l_0}x|^2 \\
&= 2|x-(K+L_\l)^{-1}K x|^2 + 2|(K+L_\l)^{-1}P_{\l_0}x|^2 \\
&\les 4\[1+\big|(K+L_\l)^{-1}\big|^2\(|K|^2+|P_{\l_0}|^2\)\]|x|^2  \\
&\les 4\[1+\big|(R_{11}+P_{\l_0})^{-1}\big|^2\(|K|^2+|P_{\l_0}|^2\)\]|x|^2.
\end{align*}
It follows that for every $\l>\l_0$,
\begin{align*}
(I + \Si_\l R_{11}) (I + \Si_\l R_{11})^\top
&= \[P_\l(R_{11}+P_\l)^{-2}P_\l\]^{-1} \\
&\ges {1\over4}\[1+\big|(R_{11}+P_{\l_0})^{-1}\big|^2\(|K|^2+|P_{\l_0}|^2\)\]^{-1} I.
\end{align*}
Letting $\l\to\i$ yields
\begin{align*}
(I + \Si R_{11}) (I + \Si R_{11})^\top \ges {1\over4}\[1+\big|(R_{11}+P_{\l_0})^{-1}\big|^2\(|K|^2+|P_{\l_0}|^2\)\]^{-1} I.
\end{align*}
This implies (a) and (b). For (c), we have from the identity
$$ \dot\Si_\l(t)P_\l(t) + \Si_\l(t)\dot P_\l(t) = {d\over dt}[\Si_\l(t)P_\l(t)] =0 $$
that
\begin{align*}
\dot\Si_\l(t)
&= -\Si_\l(t)\dot P_\l(t)\Si_\l(t) \\
&= A\Si_\l + \Si_\l A^\top - \begin{pmatrix}C^\top + S_1\Si_\l \\ B^\top + S_2\Si_\l \end{pmatrix}^\top
                             \begin{pmatrix}R_{11}+P_\l & 0 \\ 0 & R_{22} \end{pmatrix}^{-1}
                             \begin{pmatrix}C^\top + S_1\Si_\l \\ B^\top + S_2\Si_\l \end{pmatrix}  \\
&= A\Si_\l + \Si_\l A^\top - \cB(\Si_\l)R_{22}^{-1}\cB(\Si_\l)^\top - \cC(\Si_\l)(R_{11}+P_\l)^{-1}\cC(\Si_\l)^\top \\
&= A\Si_\l + \Si_\l A^\top - \cB(\Si_\l)R_{22}^{-1}\cB(\Si_\l)^\top - \cC(\Si_\l)\cR(\Si_\l)^{-1}\Si_\l\cC(\Si_\l)^\top.
\end{align*}
Consequently,
\begin{align}\label{Sil-integral}
\Si_\l(t) &= \l^{-1} I - \!\int_t^T \[A\Si_\l + \Si_\l A^\top\! -\cB(\Si_\l)R_{22}^{-1}\cB(\Si_\l)^\top\! \nn\\
&~\hp{=} -\cC(\Si_\l)\cR(\Si_\l)^{-1}\Si_\l\cC(\Si_\l)^\top\]ds.
\end{align}
Letting $\l\to\i$ in \rf{Sil-integral}, we obtain by the bounded convergence theorem that
\begin{align*}
\Si(t) &= - \int_t^T \[A\Si + \Si A^\top -\cB(\Si)R_{22}^{-1}\cB(\Si)^\top\! -\cC(\Si)^\top\cR(\Si)^{-1}\Si\cC(\Si)\]ds,
\end{align*}
which is the integral version of \rf{Ric:Sigma(Q=0)}.
\end{proof}

With the solution $\Si$ to the Riccati equation \rf{Ric:Sigma(Q=0)}, we further introduce the following linear BSDE:
\begin{equation}\label{BSDE:phi}
\left\{\begin{aligned}
  d\f(t) &= \Big\{[A -\cB(\Si)R_{22}^{-1}S_2 -\cC(\Si)\cR(\Si)^{-1}\Si S_1]\f  \\
         &~\hp{=} + \cC(\Si)\cR(\Si)^{-1}\b \Big\}dt + \b dW(t), \q t\in[0,T], \\
   \f(T) &= \xi.
\end{aligned}\right.
\end{equation}
Since $R_{22}\gg0$ and $\Si$ is such that $\cR(\Si)^{-1}\in L^\i(0,T;\dbR^n)$, the BSDE \rf{BSDE:phi} is clearly uniquely solvable.

\ms

In terms of the solution $\Si$ to the Riccati equation \rf{Ric:Sigma(Q=0)} and the adapted solution $(\f,\b)$
to the BSDE \rf{BSDE:phi}, we can construct the optimal control of Problem (BSLQ) as follows.

\begin{theorem}\label{thm:main-u}
Let \ref{ass:A1}--\ref{ass:A3} and \rf{G=Q=R12=0} hold.
Let $(\f,\b)$ be the adapted solution to the BSDE \rf{BSDE:phi} and $X$ the solution to the following SDE:
\begin{equation}\label{SDE:X-phi}
\left\{\begin{aligned}
dX(t) &= \Big\{\big[S_1^\top\cR(\Si)^{-1}\Si\cC(\Si)^\top + S_2^\top R_{22}^{-1}\cB(\Si)^\top -A^\top\big] X \\
&~\hp{=} -\big[S_1^\top\cR(\Si)^{-1}\Si S_1 + S_2^\top R_{22}^{-1}S_2\big]\f + S_1^\top\cR(\Si)^{-1}\b\Big\}dt \\
&~\hp{=} -\big[\cR(\Si)^{-1}\big]^\top\big[\cC(\Si)^\top X -S_1\f -R_{11}\b\big] dW(t), \\
X(0) &= 0.
\end{aligned}\right.
\end{equation}
Then the optimal control of Problem (BSLQ) for the terminal state $\xi$ is given by
\begin{equation}\label{opt:u-rep}\begin{aligned}
u(t) &= [R_{22}(t)]^{-1} [\cB(t,\Si(t))^\top X(t) -S_2(t)\f(t)], \q t\in[0,T].
\end{aligned}\end{equation}
\end{theorem}

\begin{proof}
Let us define for $t\in[0,T]$,
\begin{align}
\label{dcpl:Y=}
Y(t) &= -\Si(t)X(t) + \f(t), \\
\label{dcpl:Z=}
Z(t) &= \cR(t,\Si(t))^{-1}[\Si(t)\cC(t,\Si(t))^\top X(t) - \Si(t)S_1(t)\f(t) + \b(t)].
\end{align}
We observe that
\begin{align}\label{DJ=0:G=Q=0}
R_{22}u &= \cB(\Si)^\top X -S_2\f = B^\top X + S_2(\Si X-\f) = B^\top X - S_2Y.
\end{align}
Furthermore, using \rf{opt:u-rep} and \rf{dcpl:Y=}  we obtain
\begin{align*}
S_1^\top Z + S_2^\top u
&= S_1^\top\cR(\Si)^{-1}[\Si\cC(\Si)^\top X-\Si S_1\f+\b]+S_2^\top R_{22}^{-1}\cB(\Si)^\top X
  -S_2^\top R_{22}^{-1}S_2\f  \\
&= \big[S_1^\top\cR(\Si)^{-1}\Si\cC(\Si)^\top + S_2^\top R_{22}^{-1}\cB(\Si)^\top\big]X \\
&~\hp{=} -\big[S_1^\top\cR(\Si)^{-1}\Si S_1 + S_2^\top R_{22}^{-1}S_2\big]\f +S_1^\top\cR(\Si)^{-1}\b,
\end{align*}
from which it follows that
\begin{align}\label{S1Z+S2u}
-A^\top X +S_1^\top Z +S_2^\top u
&= \big[S_1^\top\cR(\Si)^{-1}\Si\cC(\Si)^\top + S_2^\top R_{22}^{-1}\cB(\Si)^\top-A^\top\big] X \nn\\
&~\hp{=} -\big[S_1^\top\cR(\Si)^{-1}\Si S_1 + S_2^\top R_{22}^{-1}S_2\big]\f + S_1^\top\cR(\Si)^{-1}\b.
\end{align}
Similarly, we can get
\begin{align*}
-C^\top X +S_1Y +R_{11}Z
&= -(C^\top+S_1\Si)X + S_1\f + R_{11}Z  \\
&= [R_{11}\cR(\Si)^{-1}\Si-I]\cC(\Si)^\top X + [I-R_{11}\cR(\Si)^{-1}\Si]S_1\f  \\
&~\hp{=} + R_{11}\cR(\Si)^{-1}\b.
\end{align*}
Noting that
\begin{align*}
& R_{11}\cR(\Si)^{-1} = R_{11}(I+\Si R_{11})^{-1} = (I+R_{11}\Si)^{-1}R_{11}, \\
& I-R_{11}\cR(\Si)^{-1}\Si = (I+R_{11}\Si)^{-1} = \big[\cR(\Si)^{-1}\big]^\top,
\end{align*}
we further obtain
\begin{align}\label{R11Z+R12u}
-C^\top X +S_1Y +R_{11}Z
&= -\big[\cR(\Si)^{-1}\big]^\top[\cC(\Si)^\top X -S_1\f -R_{11}\b].
\end{align}
This implies that the solution of \rf{SDE:X-phi} satisfies the equation
\begin{equation}\label{FBSDE:X=}
\left\{\begin{aligned}
dX(t) &= (-A^\top X +S_1^\top Z +S_2^\top u)dt +(-C^\top X +S_1Y +R_{11}Z)dW, \\
 X(0) &= 0.
\end{aligned}\right.
\end{equation}
Next, for simplicity let us set
\begin{align}\label{alpha=}
 \a=[A -\cB(\Si)R_{22}^{-1}S_2 -\cC(\Si)\cR(\Si)^{-1}\Si S_1]\f + \cC(\Si)\cR(\Si)^{-1}\b.
\end{align}
By It\^{o}'s rule, we have
\begin{align*}
dY &= -\dot\Si Xdt - \Si dX + d\f \\
&= [\a-\dot\Si X-\Si(-A^\top X +S_1^\top Z +S_2^\top u)]dt + [\b - \Si(-C^\top X +S_1Y +R_{11}Z)]dW.
\end{align*}
Using \rf{S1Z+S2u} and \rf{alpha=} we get
\begin{align*}
&\a-\dot\Si X-\Si(-A^\top X +S_1^\top Z +S_2^\top u) \\
&\q= \a-\big[\dot\Si-\Si A^\top+\Si S_1^\top\cR(\Si)^{-1}\Si\cC(\Si)^\top +\Si S_2^\top R_{22}^{-1}\cB(\Si)^\top\big]X \\
&~\hp{\q=} +\Si\big[S_1^\top\cR(\Si)^{-1}\Si S_1 + S_2^\top R_{22}^{-1}S_2\big]\f - \Si S_1^\top\cR(\Si)^{-1}\b  \\
&\q= \a-\big[A\Si-C\cR(\Si)^{-1}\Si\cC(\Si)^\top-BR_{22}^{-1}\cB(\Si)^\top\big]X \\
&~\hp{\q=} +\Si\big[S_1^\top\cR(\Si)^{-1}\Si S_1 + S_2^\top R_{22}^{-1}S_2\big]\f - \Si S_1^\top\cR(\Si)^{-1}\b  \\
&\q= AY + [C\cR(\Si)^{-1}\Si\cC(\Si)^\top + BR_{22}^{-1}\cB(\Si)^\top]X - BR_{22}^{-1}S_2\f \\
&~\hp{\q=} -C\cR(\Si)^{-1}(\Si S_1\f-\b) \\
&\q= AY + BR_{22}^{-1}[\cB(\Si)^\top X -S_2\f] + C\cR(\Si)^{-1}[\Si\cC(\Si)^\top X-\Si S_1\f+\b]  \\
&\q= AY + Bu +CZ.
\end{align*}
Using \rf{R11Z+R12u} and the relations
\begin{align*}
\Si\big[\cR(\Si)^{-1}\big]^\top &= \Si(I+R_{11}\Si)^{-1} = (I+\Si R_{11})^{-1}\Si = \cR(\Si)^{-1}\Si, \\
I-\cR(\Si)^{-1}\Si R_{11} &= I-(I+\Si R_{11})^{-1}\Si R_{11} = (I+\Si R_{11})^{-1} = \cR(\Si)^{-1},
\end{align*}
we get
\begin{align*}
&\b-\Si(-C^\top X +S_1Y +R_{11}Z) \\
&\q= \b +\Si\big[\cR(\Si)^{-1}\big]^\top[\cC(\Si)^\top X -S_1\f -R_{11}\b] \\
&\q= \cR(\Si)^{-1}[\Si\cC(\Si)^\top X -\Si S_1\f] + [I-\cR(\Si)^{-1}\Si R_{11}]\b \\
&\q= Z.
\end{align*}
Therefore, the pair $(Y,Z)$ defined by \rf{dcpl:Y=}--\rf{dcpl:Z=} satisfies the backward equation
\begin{equation}\label{FBSDE:Y=}
\left\{\begin{aligned}
dY(t) &= (AY +Bu +CZ)dt + ZdW, \\
 Y(T) &= \xi.
\end{aligned}\right.
\end{equation}
Combining \rf{DJ=0:G=Q=0}, \rf{FBSDE:X=} and \rf{FBSDE:Y=}, we see that the solution $X$ of \rf{SDE:X-phi},
the pair $(Y,Z)$ defined by \rf{dcpl:Y=}--\rf{dcpl:Z=}, and the control $u$ defined by \rf{opt:u-rep}
satisfy the FBSDE
\begin{equation}\label{FBSDE:G=Q=0}
\left\{\begin{aligned}
dX(t) &= (-A^\top X +S_1^\top Z +S_2^\top u)dt +(-C^\top X +S_1Y +R_{11}Z)dW, \\
dY(t) &= (AY +Bu +CZ)dt + ZdW, \\
 X(0) &= 0, \q Y(T) = \xi,
\end{aligned}\right.
\end{equation}
and the condition
\begin{align}\label{DJ=0:G=Q=0*}
  S_2Y - B^\top X  + R_{22}u  =0.
\end{align}
Therefore, by \autoref{thm:optimality}, $u$ is the (unique) optimal control for the terminal state $\xi$.
\end{proof}

We conclude this section with a representation of the value function $V(\xi)$.

\begin{theorem}\label{thm:main-V}
Let \ref{ass:A1}--\ref{ass:A3} and \rf{G=Q=R12=0} hold. Then the value function of Problem (BSLQ) is given by
\begin{align}\label{V(xi)=}
V(\xi) &= \dbE\int_0^T \Big\{\lan R_{11}\cR(\Si)^{-1}\b,\b\ran + 2\lan S_1^\top\cR(\Si)^{-1}\b,\f\ran \nn\\
&~\hp{=} -\lan[S_1^\top\cR(\Si)^{-1}\Si S_1+S_2^\top R_{22}^{-1}S_2]\f,\f\ran \Big\}dt,
\end{align}
where $(\f,\b)$ is the adapted solution to the BSDE \rf{BSDE:phi}.
\end{theorem}

\begin{proof}
Let $u$ be the optimal control for the terminal state $\xi$. Then, by \autoref{thm:optimality},
the adapted solution $(X,Y,Z)$ of \rf{FBSDE:G=Q=0} satisfies \rf{DJ=0:G=Q=0*}. Observe that
\begin{align*}
V(\xi) &= J(\xi;u) = \dbE\int_0^T \[2\lan S_1Y,Z\ran + 2\lan S_2Y,u\ran + \lan R_{11}Z,Z\ran + \lan R_{22}u,u\ran\]dt \\
&= \dbE\int_0^T \[\lan S_1^\top Z+S_2^\top u,Y\ran + \lan S_1Y+R_{11}Z,Z\ran + \lan S_2Y+R_{22}u,u\ran\]dt \\
&= \dbE\int_0^T \[\lan S_1^\top Z+S_2^\top u,Y\ran + \lan S_1Y+R_{11}Z,Z\ran + \lan B^\top X,u\ran\]dt.
\end{align*}
Integration by parts yields
\begin{align*}
\dbE\lan X(T),Y(T)\ran
&= \dbE\int_0^T\[\lan X,AY+Bu+CZ\ran + \lan-A^\top X+S_1^\top Z+S_2^\top u,Y\ran \\
&~\hp{=} +\lan-C^\top X+S_1Y+R_{11}Z,Z\ran\]dt \\
&= \dbE\int_0^T \[\lan X,Bu\ran + \lan S_1^\top Z+S_2^\top u,Y\ran + \lan S_1Y+R_{11}Z,Z\ran \]dt \\
&= V(\xi).
\end{align*}
From the proof of \autoref{thm:main-u}, we see that $X$ also satisfies the equation \rf{SDE:X-phi}.
Using \rf{SDE:X-phi} and integration by parts again, we obtain
\begin{align*}
\dbE\lan X(T),\f(T)\ran
&= \dbE\int_0^T\Big\{\blan \big[S_1^\top\cR(\Si)^{-1}\Si\cC(\Si)^\top + S_2^\top R_{22}^{-1}\cB(\Si)^\top -A^\top\big]X,\f\bran \\
&~\hp{=} -\blan\big[S_1^\top\cR(\Si)^{-1}\Si S_1 + S_2^\top R_{22}^{-1}S_2\big]\f,\f\bran + \blan S_1^\top\cR(\Si)^{-1}\b,\f\bran  \\
&~\hp{=} +\blan X,[A -\cB(\Si)R_{22}^{-1}S_2 -\cC(\Si)\cR(\Si)^{-1}\Si S_1]\f\bran + \blan X,\cC(\Si)\cR(\Si)^{-1}\b\bran \\
&~\hp{=} -\blan\big[\cR(\Si)^{-1}\big]^\top[\cC(\Si)^\top X -S_1\f -R_{11}\b],\b\bran \Big\}dt \\
&= \dbE\int_0^T\Big\{\lan R_{11}\cR(\Si)^{-1}\b,\b\ran + 2\lan S_1^\top\cR(\Si)^{-1}\b,\f\ran  \\
&~\hp{=} -\lan[S_1^\top\cR(\Si)^{-1}\Si S_1+S_2^\top R_{22}^{-1}S_2]\f,\f\ran \Big\}dt.
\end{align*}
The representation \rf{V(xi)=} then follows from the fact
$$ V(\xi) = \dbE\lan X(T),Y(T)\ran = \dbE\lan X(T),\xi\ran = \dbE\lan X(T),\f(T)\ran. $$
The proof is complete.
\end{proof}

\section{Conclusion}\label{sec:conclusion}

For the reader's convenience, we conclude the paper by generalizing the results obtained in \autoref{sec:construction}
to the case without the assumption \rf{G=Q=R12=0}.
We shall only present the result, as the proof can be easily given using the argument at the beginning
of \autoref{sec:construction} and the results established there for the case \rf{G=Q=R12=0}.

\ms

Recall the notation
\begin{align*}
\sC(t)      &= C(t)-B(t)[R_{22}(t)]^{-1}R_{21}(t), \\
\sS_1(t)    &= S_1(t)-R_{12}(t)[R_{22}(t)]^{-1}S_2(t), \\
\sR_{11}(t) &= R_{11}(t)-R_{12}(t)[R_{22}(t)]^{-1}R_{21}(t).
\end{align*}
Let $H\in C([0,T];\dbS^n)$ be the unique solution to the linear ODE
\begin{equation*}
\left\{\begin{aligned}
  &\dot H(t) + H(t)A(t) + A(t)^\top H(t) + Q(t) =0, \q t\in[0,T],\\
  &H(0)= G,
\end{aligned}\right.
\end{equation*}
and let
$$\begin{aligned}
S_{^1}^{_H}(t)    &= \sS_1(t) + \sC(t)^\top H(t),  \q& \cB^{_H}(t,\Si(t)) &= B(t)+\Si(t)[S_{^2}^{_H}(t)]^\top,\\
S_{^2}^{_H}(t)    &= S_2(t)+B(t)^\top H(t),        \q& \cC^{_H}(t,\Si(t)) &= \sC(t)+\Si(t)[S_{^1}^{_H}(t)]^\top,\\
R_{^{11}}^{_H}(t) &= \sR_{11}(t)+H(t),             \q& \cR^{_H}(t,\Si(t)) &= I + \Si(t)R_{^{11}}^{_H}(t).
\end{aligned}$$

\begin{theorem}
Let \ref{ass:A1}--\ref{ass:A3} hold. We have the following results.
\begin{enumerate}[(i)]
\item The Riccati equation
\begin{equation}\label{Ric:general}
\left\{\begin{aligned}
&\dot\Si(t) - A(t)\Si(t) - \Si(t)A(t)^\top + \cB^{_H}(t,\Si(t))[R_{22}(t)]^{-1}[\cB^{_H}(t,\Si(t))]^\top \\
&\hp{\dot\Si(t)} +\cC^{_H}(t,\Si(t))[\cR^{_H}(t,\Si(t))]^{-1}\Si(t)[\cC^{_H}(t,\Si(t))]^\top=0, \\
&\Si(T)= 0
\end{aligned}\right.
\end{equation}
admits a unique positive semidefinite solution $\Si\in C([0,T];\dbS^n)$ such that $\cR^{_H}(\Si)$ is invertible
a.e. on $[0,T]$ and $[\cR^{_H}(\Si)]^{-1}\in L^\i(0,T;\dbR^n)$.
\item Let $(\f,\b)$ be the adapted solution to the BSDE
\begin{equation}\label{BSDE:general}
\left\{\begin{aligned}
  d\f(t) &= \Big\{[A -\cB^{_H}(\Si)R_{22}^{-1}S_2^{_H} -\cC^{_H}(\Si)[\cR^{_H}(\Si)]^{-1}\Si S_1^{_H}]\f  \\
         &~\hp{=} + \cC^{_H}(\Si)[\cR^{_H}(\Si)]^{-1}\b \Big\}dt + \b dW(t), \q t\in[0,T], \\
   \f(T) &= \xi,
\end{aligned}\right.
\end{equation}
and let $X$ be the solution to the following SDE:
\begin{equation*}
\left\{\begin{aligned}
dX(t) &= \Big\{\[(S_{^1}^{_H})^\top[\cR^{_H}(\Si)]^{-1}\Si[\cC^{_H}(\Si)]^\top + (S_{^2}^{_H})^\top R_{22}^{-1}[\cB^{_H}(\Si)]^\top -A^\top\] X \\
&~\hp{=} -\[(S_{^1}^{_H})^\top[\cR^{_H}(\Si)]^{-1}\Si S_{^1}^{_H} + (S_{^2}^{_H})^\top R_{22}^{-1}S_{^2}^{_H}\]\f + (S_{^1}^{_H})^\top[\cR^{_H}(\Si)]^{-1}\b\Big\}dt \\
&~\hp{=} -\big[\cR^{_H}(\Si)^{-1}\big]^\top[\cC^{_H}(\Si)^\top X -S_{^1}^{_H}\f -R_{^{11}}^{_H}\b]dW, \\
X(0) &= 0.
\end{aligned}\right.
\end{equation*}
Then the optimal control of Problem (BSLQ) for the terminal state $\xi$ is given by
\begin{align}\label{opt:general}
u &= R_{22}^{-1}\Big\{ [\cB^{_H}(\Si)^\top-R_{21}\cR^{_H}(\Si)^{-1}\Si\cC^{_H}(\Si)^\top]X  \nn\\
&~\hp{=} +[R_{21}\cR^{_H}(\Si)^{-1}\Si S_{^1}^{_H}-S_{^2}^{_H}]\f - R_{21}\cR^{_H}(\Si)^{-1}\b\Big\}.
\end{align}
\item The value function of Problem (BSLQ) is given by
\begin{align}\label{V(xi):general}
V(\xi) &= -\,\dbE\lan H(T)\xi,\xi\ran + \dbE\int_0^T \Big\{\blan R_{^{11}}^{_H}\cR^{_H}(\Si)^{-1}\b,\b\bran
          +2\blan (S_{^1}^{_H})^\top\cR^{_H}(\Si)^{-1}\b,\f\bran \nn\\
&~\hp{=}  -\blan[(S_{^1}^{_H})^\top\cR^{_H}(\Si)^{-1}\Si S_{^1}^{_H}+ (S_{^2}^{_H})^\top R_{22}^{-1}S_{^2}^{_H}]\f,\f\bran \Big\}dt ,
\end{align}
where $(\f,\b)$ is the adapted solution to the BSDE \rf{BSDE:general}.
\end{enumerate}
\end{theorem}

To summarize, we have investigated an indefinite backward stochastic LQ optimal control problem with deterministic
coefficients and have developed a general procedure for constructing optimal controls.
The crucial idea is to establish the connection between backward stochastic LQ optimal control problems and forward
stochastic LQ optimal control problems (see \autoref{sec:connection}) and to convert the backward stochastic LQ optimal
control problem into an equivalent one for which the limiting procedure applies (see \autoref{sec:construction}).
The results obtained in the paper provide insight into some related topics, especially into the study of zero-sum
stochastic differential games (as mentioned in the introduction).
We hope to report some relevant results along this line in our future publications.

\end{document}